\documentclass[12pt,letterpaper]{amsart}
\usepackage[pass]{geometry}
\usepackage{amssymb,amsmath,amsthm,amsgen}
\usepackage{tikz}
\usepackage{comment}
\usepackage{}
\usepackage{enumitem}
\newcommand{\old}[1]{}
\theoremstyle{plain}

\newtheorem{thm}{Theorem}
\newtheorem{lem}[thm]{Lemma}
\newtheorem{conj}{Conjecture}
\newtheorem{cor}[thm]{Corollary}

\newtheorem{prop}[thm]{Proposition}
\theoremstyle{definition}
\newtheorem{defn}[thm]{Definition}

\newtheorem{ex}[thm]{Example}
\newtheorem{rk}[thm]{Remark}
\newtheorem{qn}[conj]{Question}

\numberwithin{equation}{section}

 at 10truept

\title{The reflection representation in the homology  of subword order}

\author{Sheila Sundaram}
\address{Pierrepont School, One Sylvan Road North, Westport, CT 06880}
\email{shsund@comcast.net}
\date{\today}

\thanks{The author is grateful to the anonymous reviewers for their detailed and valuable comments.}


\begin{document}
\begin{abstract}  We investigate the homology representation of the symmetric group on rank-selected subposets of subword order.
  We show that the homology module for  words of bounded length, over an alphabet of size $n,$ 
  decomposes into a sum of tensor powers of the $S_n$-irreducible $S_{(n-1,1)}$ indexed by the partition $(n-1,1),$  recovering, as a special case,  a theorem of Bj\"orner and Stanley  for words of length at most $k.$   For arbitrary ranks we show that the  homology  is an integer combination of  positive tensor powers of the reflection representation $S_{(n-1,1)}$,  and conjecture that this combination is nonnegative.  We uncover a curious duality in homology in the case when one rank is deleted. 

We prove that the action on the rank-selected chains of subword order is a nonnegative integer combination of tensor powers of $S_{(n-1,1)}$, and  show that its Frobenius characteristic is $h$-positive and supported on the set 
$T_{1}(n)=\{h_\lambda: \lambda=(n-r, 1^r), r\ge 1\}.$

 Our most definitive result describes the Frobenius characteristic of the homology for an arbitrary set of ranks, plus or minus one copy of  the Schur function $s_{(n-1,1)},$  as an integer combination of the set 
 $T_{2}(n)=\{h_\lambda: \lambda=(n-r, 1^r), r\ge 2\}.$  We conjecture that 
 this combination is nonnegative, establishing this fact for particular cases.  
 
\end{abstract}
\maketitle
\section{Introduction}

Let $A^*$ denote the free monoid of words of finite length in an alphabet $A.$ 
Subword order is defined on $A^*$ by setting $ u\le v$ if $u$ is a subword of $v,$ that is, 
the word $u$ is obtained by deleting letters of the word $v.$  This makes  $(A^*, \leq)$ into a graded poset 
with rank function given by the length  $|w|$ of a word $w,$ the number of letters in $w$.  The topology of this poset was first studied  by Farmer (1979) 
and then by Bj\"orner, who showed in \cite[Theorem 3]{Bj1} that any interval of this poset admits a dual CL-shelling. The intervals  are thus homotopy Cohen-Macaulay, as well as  all rank-selected subposets  
obtained by considering only words whose rank belongs to
a finite set S \cite[Theorem~4.1]{BjShellableTAMS}, \cite[Theorem~8.1]{BjWachs}.   
Suppose now that the alphabet A is finite, of cardinality $n$. The symmetric group
$S_n $ acts  on $A$, and thus on $A^*.$  To avoid trivialities we will assume $n\ge 2.$

In this paper we describe  the homology representation 
 of  intervals $[r, k]$ of consecutive ranks in $A^*$, as well as some other rank-selected subposets, using the Whitney homology technique and other methods developed in \cite{Su0}.  All homology in this paper is taken over the field of complex numbers.  
  We refer the reader to \cite{St1} for general facts about rank-selection.  
We show that  the unique nonvanishing 
homology of the rank-selected subposet $A^*_{[r,k]}$ 
 decomposes as  a direct sum of copies of $r$ consecutive tensor powers of the reflection 
representation of $S_n$, that is, the irreducible representation $S_{(n-1,1)}$ indexed by the partition $(n-1,1).$   Theorem~\ref{Bordeaux1991mai} on consecutive ranks generalises a theorem in \cite{Bj1} (conjectured by Bj\"orner and proved by Stanley) on the homology representation of the poset of all words of length at most $k.$ We establish similar results for the Whitney and dual Whitney homology modules.   Both turn out to be  permutation modules in each degree, with pleasing  orbit stabilisers. Theorem~\ref{rank-deletion} establishes the nonnegativity property with respect to tensor powers of $S_{(n-1,1)}$   for the case when one rank is deleted from the interval $[1,k],$  and leads to   a curious  homology isomorphism (Proposition~\ref{rank-deletion-duality}), suggesting a homotopy equivalence between the simplicial complexes associated to the rank sets $[1,k]\backslash{\{r\}}$ and $[1,k]\backslash{\{k-r\}},$ for fixed $r, 1\le r\le k-1.$  Finally Theorem~\ref{TwoRanks} establishes that the homology is a nonnegative sum of tensor powers of $S_{(n-1,1)}$ for rank-sets of size 2.

 More generally, we show in Theorem~\ref{RankSelectedChains} that for any nonempty subset $S$ of ranks $[1,k],$ the homology representation of $S_n$ may be written as an integer combination of positive tensor powers of the reflection representation.  We propose the following conjecture, which is supported by  Theorems~\ref{Bordeaux1991mai}, ~\ref{rank-deletion} and ~\ref{TwoRanks}:

\begin{conj}\label{Conj:tensorpower}  Let $A$ be an alphabet of size $n\ge 2.$ Then the $S_n$-homology  module of any finite nonempty rank-selected subposet of subword order on $A^*$ is a \textit{nonnegative} integer combination of positive tensor powers of the irreducible indexed by the partition $(n-1,1).$ 
\end{conj}

These considerations lead us to examine the tensor powers of the reflection representation (see Section~\ref{PowersOfRefl}), and the question of how many tensor powers are linearly independent characters.  In answering these questions, we are led to a decomposition (Theorem~\ref{ReflRepMain}) showing that the $k$th tensor power of $S_{(n-1,1)}$ plus or minus one copy of $S_{(n-1,1)},$  has Frobenius characteristic equal to a nonnegative integer combination of the homogeneous  symmetric functions $\{h_{(n-r, 1^r)}: r\ge 2\}.$  It is \lq\lq almost" an $h$-positive  permutation module.   (In general the homology itself is \textit{not} a permutation module.) Inspired by this phenomenon, we prove, in Theorem~\ref{Homologyhbasis}, that in fact for all rank subsets $T,$
the homology module $\tilde{H}(T)$ 
has the property that $\tilde{H}(T) + (-1)^{|T|} S_{(n-1,1)} $ has Frobenius characteristic equal to an \textit{integer} combination of the homogeneous symmetric functions $\{h_{(n-r, 1^r)}: r\ge 2\}.$  Theorem~\ref{AlmosthpositiveInstances} establishes the truth of the following conjecture for 
the homology of several subsets of ranks. 
\begin{conj}\label{Conj:almost-h-pos} Let $A$ be an alphabet of size $n\ge 2.$ Then the  $S_n-$homology  module of any finite nonempty rank-selected subposet of subword order on $A^*$, plus or minus one copy of the reflection representation of $S_n,$ is a permutation module. In fact its Frobenius characteristic is $h$-positive and supported on the set  $T_{2}(n)=\{h_\lambda: \lambda=(n-r, 1^r), r\ge 2\}.$
\end{conj}

We give a simple criterion for when  Conjecture~\ref{Conj:tensorpower} will imply Conjecture~\ref{Conj:almost-h-pos} in Lemma~\ref{RefRephPos}.

The  main results of this paper are summarised below.  Let $A$ be an alphabet of size $n\ge 2,$ and $T\subseteq [1,k]$ a subset of ranks.
\begin{thm}\label{SummaryChains}  The $S_n$-module induced by the action of $S_n$ on the maximal chains of the rank-selected subposet of $A^*$ of words with lengths in $T,$  is a \textbf{nonnegative} integer combination of tensor powers of the reflection representation $S_{(n-1,1)}.$  If $|T|\ge 1,$  this module  has $h$-positive Frobenius characteristic supported on the set $T_{1}(n)=\{h_\lambda: \lambda=(n-r, 1^r), r\ge 1\}.$ 
\end{thm}

\begin{thm}\label{Summary}  The homology module $\tilde{H}(T)$ of words with lengths in $T$ is an integer combination of positive tensor powers of the reflection representation $S_{(n-1,1)},$
with the property that $\tilde{H}(T) + (-1)^{|T|} S_{(n-1,1)} $ has Frobenius characteristic equal to an integer combination of the homogeneous symmetric functions $\{h_{(n-r, 1^r)}: r\ge 2\}.$ 

Both  integer combinations  are  \textbf{nonnegative} 
when $T$ is one of the following  rank sets:   
$ (1)\ [r,k], k\ge r\ge 1; \quad (2)\ [1,k]\backslash\{r\}, k\ge r\ge 1;\quad  
(3)\ \{1\le s_1<s_2\le k\}.$
 \end{thm}

By standard symmetric function theory,  any $S_n$-module is a (virtual) sum of permutation modules whose point-stabilisers are Young subgroups.  Thus a salient feature of Theorems~\ref{SummaryChains} and ~\ref{Summary} is that only Young subgroups  indexed by  hooks appear.

\section{Subword order}

The subword order poset $A^*$ has a unique least element at rank 0, namely the empty word $\emptyset$ of length zero.
In this section we collect the main facts on subword order from \cite{Bj1} that we will need.  For general facts about posets, M\"obius functions, etc. we refer the reader to \cite{St3EC1}.

\begin{defn}\cite{F} A word $\alpha$ in $A^*$ is  \textit{normal} if no two consecutive letters of $\alpha$ are equal.  
\end{defn}
For example, $aabbccaabbcc$ is not normal, while $abcabc$ is normal.  Normal words are also called \textit{Smirnov} words in the recent literature.   Observe that the number of normal words of length $i$ is 
$n(n-1)^{i-1}.$ 

\begin{thm}\label{Farmer} (Farmer \cite{F})
\begin{enumerate}
\item
Let $\alpha$ be any word in $A^*$.  Then the M\"obius function of subword order satisfies $\mu(\hat 0, \alpha)
=\begin{cases} (-1)^{|\alpha|}, & \text{ if }\alpha \text{ is a normal word}\\
0, & \text{ otherwise}.
\end{cases}$
\item (See also \cite{V}.) Let $|A|=n$ and let $A^*_{n,k}$ denote the subposet of $A^*$ consisting of the first $k$ nonzero ranks and the empty word, i.e. of words of length at most $k,$ with an artificially appended top element $\hat 1.$ Then 
\begin{equation}\label{MobiusNumber}\mu(A^*_{n,k})=\mu(\hat 0, \hat 1) = (-1)^{k-1} (n-1)^k.\end{equation}
\item \cite[Theorem 5 and preceding Remark]{F} $A^*_{n,k}$ has the homology of a wedge of $(n-1)^k$ spheres of dimension $(k-1).$
\end{enumerate}
\end{thm}

Bj\"orner generalised Part (1) above to give a simple formula for the M\"obius function of an arbitrary interval $(\beta, \alpha)$, as follows.  

\begin{defn} \cite{Bj1}
Given a word $\alpha= a_1a_2\ldots a_n$ in $A^*$, its repetition set is $R(\alpha)=\{i: a_{i-1}=a_i\}.$   An embedding of $\beta$ in $\alpha$ is a sequence $1\le i_1<i_2<\ldots<i_k\le n$ such that $\beta =a_{i_1} a_{i_2}\ldots a_{i_k}.$ It is called a \textit{normal embedding} if in addition $R(\alpha)\subseteq \{i_1,i_2,\ldots,i_k\}.$

Denote by ${\alpha \choose \beta}$ the number of embeddings of $\beta$ in $\alpha,$ and by ${\alpha \choose \beta}_n$ the number of normal embeddings of $\beta$ in $\alpha$. 

\end{defn}

\begin{thm} \cite[Theorem 1]{Bj1} For all $\alpha, \beta\in A^*,$ 
\[\mu(\beta,\alpha)=(-1)^{|\alpha|-|\beta|} {\alpha \choose \beta}_n.\]
\end{thm}
Observing that the word $\alpha$ is normal if and only if its repetition set $R(\alpha)$ is empty, one sees that this generalises  Farmer's formula for $\mu(\hat 0, \alpha).$

Recall that the zeta function \cite{St3EC1} of a poset is defined by 
$\zeta(\beta, \alpha)=1$ if $\beta\le \alpha,$ and equals zero otherwise. 
\begin{thm}\label{Bjgf}\cite{Bj1} Let $A$ be an alphabet of size $n,$ and $\beta$  a word in $A^*$ of length $k.$  The following generating functions hold:
\begin{enumerate}
\item \cite[Theorem~2~(i)]{Bj1} 
 For the  M\"obius function of subword order: 
\[\sum_{\alpha\in A^*} \mu(\beta,\alpha)t^{|\alpha|} 
= \dfrac{t^k(1-t)}{(1+(n-1)t)^{k+1}}.\]
\item \cite[3.~Remark.]{Bj1}   The number of words of length $p$ in the interval $[\beta, \infty]$ depends only on  
the length $k$ of $\beta, $ and equals 
\[\sum_{i=0}^{p-k} {p\choose i} (n-1)^i.\]

\item \cite[3.~Remark.(i)]{Bj1} For the zeta function of subword order:
\[\sum_{\alpha\in A^*} \zeta(\beta,\alpha)t^{|\alpha|} 
= \dfrac{t^k}{(1-nt)(1-(n-1)t)^{k}}.\]
\end{enumerate}
\end{thm}

Farmer's result on the homology of $A^*_{n,k}$   was strengthened by Bj\"orner, who  showed the following (see \cite{BjShellableTAMS}, \cite{BjWachs}, and also \cite{Wa} for a survey of lexicographic shellability):
\begin{thm}\label{BjCLshellable}(Bj\"orner \cite[Theorem 3, Corollary 2]{Bj1}) Every interval $(\beta, \alpha)$ in the subword order poset $A^*$ is dual CL-shellable, and hence homotopy Cohen-Macaulay.  In particular, for a finite alphabet $A,$ the poset $A^*_{n,k}$ of nonempty words of length at most $k,$ which may be viewed as the result of rank-selection from an appropriate interval of $A^*$, is also dual CL-shellable and hence also homotopy Cohen-Macaulay. 
\end{thm}

We point out three  details about Farmer's original paper:
 \begin{description}
 \item[a]    The order used in the present paper is what Farmer calls the \textit{embedding} order (see \cite[p.609]{F}).  Farmer's \lq\lq subword order" differs from ours and \cite{Bj1}.
 \item[b] All homology in Farmer's paper is ordinary homology, as opposed to reduced homology  in the present paper and \cite{St1}. 
 In keeping with his definition of a graded poset, for the rank function $d$ 
 of subword order, Farmer defines $d(w)=j-1$ if $w$ is  a word of length $j;$ in this paper we use the length of the word as its rank.
 \item[c]  In particular, his definition of the $k$-skeleton $X^k$ of a poset of words $X$ corresponds to our $A^*_{k+1},$ i.e. to taking the words of length at most $(k+1)$.
 \end{description}

\section{Rank-selection in $A^*$}

In this section we will assume the alphabet $A$ is finite of size $n.$ 

We follow the standard convention as in \cite{St1}, \cite{St3EC1}: By the homology of a poset $P$ with greatest element $\hat 1$ and least element $\hat 0,$ we mean the reduced homology $\tilde{H}(P)$ of the simplicial complex whose faces are the chains of 
$P\backslash\{\hat 0, \hat 1\}.$
In order to determine the homology of rank-selected subposets of $A^*_{n,k},$ we will use the techniques developed in \cite{Su0}.  For an elementary treatment of these and more general methods, see \cite{SuJer}. 

Whitney homology was originally defined by Baclawski \cite{BaWhitneyHom}.   Bj\"orner showed \cite{Bj0WhitneyHom} that  the $i$th Whitney homology of a graded Cohen-Macaulay 
poset $P$ with  least  element $\hat 0$ is given by the isomorphism
\begin{equation}\label{WHBjdef}W\!H_i(P)\simeq\bigoplus_{x:\mathrm{rank}(x)=i} \tilde{H}_{i-2}(\hat 0, x).\end{equation}
Note that if $P$ has a top element $\hat 1,$ then the top Whitney homology coincides with the top homology of $P.$

If $G$ is a group acting on the Cohen-Macaulay poset $P$, then $W\!H_i(P)$ is also a 
$G$-module.
The present author observed  that the isomorphism~(\ref{WHBjdef})  is in fact group-equivariant, and also established the  equivariant acyclicity of Whitney homology   (see \cite{Su0}).  Thus  
(\ref{WHBjdef}) becomes an effective tool for computing both the $G$-module structure of Whitney homology as well as the homology of the full poset, as an equivariant analogue of the inherently recursive structure in the M\"obius function.  This technique was then  exploited 
in \cite{Su0} and \cite{SuJer} to determine group actions on the homology of  posets.  

It is also computationally useful to consider the \textit{dual} Whitney homology of the Cohen-Macaulay poset $P$ when $P$ has a top element $\hat 1$, that is, the Whitney homology of the dual poset $P^*$, which we denote by $W\!H^*(P).$ Note that we now have an equivariant isomorphism 
\begin{equation}\label{WHdualBjdef}W\!H^*_i(P)\simeq\bigoplus_{x:\mathrm{rank}(x)=r-i} \tilde{H}_{i-2}( x, \hat 1), 0\leq i\leq r.\end{equation}
Here $r$ is the length of the longest chain from $\hat 0$ to $\hat 1.$

See \cite{Su0} and \cite{SuJer} for a more general version of the following  theorem (for arbitrary posets), and also   \cite{Wa} for additional background on Whitney homology. 
\begin{thm}\label{WHmainSS}\cite[Lemma~1.1, Theorem 1.2, Proposition~1.9]{Su0} Let $P$ be a graded Cohen-Macaulay poset of rank $r$  carrying an action of a group $G.$ Then the unique nonvanishing top homology of $P$ coincides with the top Whitney homology module $W\!H_r(P),$ and  as a $G$-module, can be computed as an alternating sum of Whitney homology modules:
\begin{equation}\label{Acyclicity}\tilde{H}_{r-2}(P)\simeq \bigoplus_{i=0}^{r-1} (-1)^i W\!H_{r-1-i}(P).\end{equation}
In particular, if $P(\underline{k})$ denotes the subposet consisting of the first $k$ nonzero ranks, with a bottom and top element attached, then one has the $G$-module decomposition 
\begin{equation}\label{ConsecRanks}
\tilde{H}_{k-2}(P(\underline{k-1}))\oplus \tilde{H}_{k-1} (P(\underline{k}))
\simeq W\!H_k(P), r\geq k\geq 1.\end{equation}
Note that $W\!H_0(P)$ is the trivial $G$-module, while $W\!H_r(P)$ gives the reduced top homology of the poset $P$.
\end{thm}

Richard Stanley proved Bj\"orner's conjecture that 
\begin{thm}\label{BjSt} \cite[Theorem 4]{Bj1} The $S_n$-action on the unique nonvanishing homology $\tilde H_{k-1}(A^*_{n,k}, \mathbb C)$ is the module given by the $k$th tensor power of the irreducible representation indexed by the partition $(n-1,1).$ 
\end{thm}

In the following theorem, we formalise Stanley's insight into subword order,  as  used in the proof of the above theorem. 
For more background on the Hopf trace formula and its use in poset homology, see \cite{SuJer}. Recall that the \textit{Lefschetz module} of a poset $P$ is the alternating sum (by degree) of the homology modules of (the order complex) of $P.$

Denote by $S_\lambda$ the irreducible representation of the symmetric group $S_n$ indexed by the partition $\lambda$ of $n$, and write $S_\lambda
^{\otimes i}$ for the $i$th tensor power of the module $S_\lambda$.

\begin{thm}\label{BjStproof}  Let $\{P_n\}$ be any sequence of finite posets each carrying an action of the symmetric group $S_n,$ such that 
\begin{enumerate} 
\item For any $g\in S_n,$ the fixed-point subposet $P_n^g$ is isomorphic to the poset $P_{\mathrm {fix}(g)},$  where $\mathrm {fix}(g)$ is the number of fixed points of $g$ as a permutation of $S_n$, 
and
\item the M\"obius number $\mu(P_n)$ is a polynomial in $(n-1),$  say $\sum_{i\geq 0} b_i (n-1)^i. $ 
\end{enumerate}
Then the  Lefschetz module of $P_n$ decomposes  as a sum of $i$th tensor powers of  the irreducible indexed by the partition $(n-1,1),$ with coefficient equal to $b_i, i\ge 0.$  (Note that the $0$th tensor power corresponds to the trivial $S_n$-module $S_{(n)}.$) In particular, the $S_n$-module structure of the Lefschetz module of $P_n$ is completely determined by its M\"obius number.

 Now assume $P_n$ is Cohen-Macaulay, as well as all the fixed-point subposets $P_n^g.$  If for all $k\geq 0,$  the Betti number of $W\!H_k(P_n)$ is a polynomial in $(n-1),$ then this polynomial determines the trace of $g\in S_n$ on the $k$th Whitney homology of $P_n.$ The representation  of $S_n$ on $W\!H_k(P_n)$ is therefore a linear combination of tensor powers of the irreducible $S_{(n-1,1)}.$ 
\end{thm}

\begin{proof}  This is clear since 
\begin{enumerate} 
\item (\cite{St1}, \cite{St3EC1}, \cite{SuJer}) the Lefschetz module of $P_n$ has (virtual) degree  $\mu(P_n),$  the Euler characteristic of the order complex of $P_n;$  
\item (\cite{St3EC1}, \cite{SuJer}) by the Hopf trace formula, the trace of an element $g\in S_n$ on this Lefschetz module is the  M\"obius number $\mu(P^g_n)$ of the fixed-point poset $P_n^g,$   since it is the Euler characteristic of the order complex of $P_n^g$; 
\item by hypothesis, $\mu(P^g_n)=\mu(P_{\mathrm{ fix}(g)})=\sum_i b_i (\mathrm{ fix}(g)-1)^i,$  and finally 
\item the trace of $g$ on the irreducible $S_n$-module indexed by $(n-1,1)$ is $\mathrm{ fix}(g)-1.$
\end{enumerate}

Similar conclusions hold for  Whitney homology in the case when the posets are Cohen-Macaulay. The key observation here is that from Bj\"orner's formulation eqn.~\eqref{WHBjdef}, it follows that the Whitney homology of the fixed-point subposet $P_n^g$ coincides with the Whitney homology of $P_{\mathrm{fix}(g)}.$
\end{proof}

Our motivating example for the poset $P_n$ satisfying the conditions of Theorem~\ref{BjStproof} is clearly subword order $A^*$ when $|A|=n.$
More generally, fix an integer $k\geq 1,$ and let $S$ be any subset of the ranks $[1,k]$.  Then the rank-selected subposet $A^*_{n,k}(S)$ of $A^*$ consisting of elements with ranks belonging to $S$ also satisfies the conditions of Theorem~\ref{BjStproof}.  When $S=[1,k]$ we denote this rank-selected subposet $A^*_{[1,k]}$ simply by $A^*_{n,k}.$

Using the generating function for the M\"obius function of $A^*$ given in Theorem~\ref{Bjgf}, Theorem~\ref{WHsubword} below computes all but the top Whitney homology $S_n$-modules for subword order.  The proof  requires a key formula, which we derive from the generating function for the M\"obius function of $A^*$ given in Theorem~\ref{Bjgf}.  We  isolate this computation in the following lemma.

\begin{lem} \label{UpperIntSubwordOrder} Let $\beta$ be any element of $A^*_{n,k}\backslash\{\hat 1\},$ where the alphabet $A$ has cardinality $n.$ Then 
\[\mu(\beta, \hat 1)_{A^*_{n,k}}(-1)^{k+1-|\beta|} 
= {k\choose |\beta|} (n-1)^{k-|\beta|}.\]
In particular this M\"obius number depends only on the rank (length) of the word $\beta.$
\end{lem}
\begin{proof}  For convenience let $|\beta|=i.$ We have, using the defining recurrence for the M\"obius function and the generating function in (1) of Theorem~\ref{Bjgf},
\begin{align*} &\mu(\beta, \hat 1)_{A^*_{n,k}}(-1)^{k+1-|\beta|}
=(-1)^{k+1-i} (-1) \sum_{\stackrel{\alpha\in A^*_{n,k}}{\alpha<\hat 1}}\mu(\beta,\alpha)
=(-1)^{k-i}\sum_{j=i}^k \sum_{\stackrel{\alpha\in A^*_{n,k}}{ |\alpha|=j}} \mu(\beta, \alpha)\\
&=  (-1)^{k-i}\sum_{j=i}^k [t^j] (1-t) t^i (1+t(n-1))^{-(i+1)}
=(-1)^{k-i}\sum_{j=i}^k [t^{j-i}] (1-t) (1+t(n-1))^{-(i+1)}.
\end{align*}
Setting $u=j-i,$ this in turn equals
\begin{align*}
&(-1)^{k-i}\sum_{u=0}^{k-i} [t^{u}] (1-t) (1+t(n-1))^{-(i+1)}
=(-1)^{k-i}[t^{k-i}](1+t(n-1))^{-(i+1)},\\
& (\text{since for any power series }f(t), \text{ one has }\sum_{j=0}^m[t^j](1-t)f(t)=[t^m]f(t)),\\
&=(-1)^{k-i}\binom{-(i+1)}{k-i}(n-1)^{k-i}=\binom{i+1+k-i-1}{k-i} (n-1)^{k-i}.
\end{align*}
The last line follows since ${-m\choose j}=(-1)^j{m+j-1\choose j}$, thereby completing the proof.\end{proof}

\begin{thm}\label{WHsubword} Consider the subword order poset $A^*_{n,k},$ with $|A|=n.$ As $S_n$-modules, the Whitney homology $W\!H(A^*_{n,k})$  and the dual Whitney homology $W\!H^*(A^*_{n,k}),$ 
for $1\le i\le k,$ are as follows. Note that $W\!H_0(A^*_{n,k})=S_{(n)}=W\!H_{k+1}^*(A^*_{n,k})$ (the trivial $S_n$-module).
\begin{equation}\label{Whitney}
W\!H_{i}(A^*_{n,k})=S_{(n-1,1)}^{\otimes i} \oplus S_{(n-1,1)}^{\otimes (i-1)};\end{equation}
\begin{align}\label{dualWhitney} W\!H^*_{k+1-i}(A^*_{n,k})&={k\choose i} S_{(n-1,1)}^{\otimes (k-i)} \otimes (S_{(n-1,1)}\oplus S_{(n)})^{\otimes i}\\
&=  \bigoplus_{j=0}^i{k\choose i}{i\choose j} S_{(n-1,1)}^{\otimes j+(k-i)}.\end{align}
\end{thm}

\begin{proof} Equation~\eqref{Whitney} is immediate from 
Theorems~\ref{BjSt} and ~\ref{WHmainSS}. 

For fixed $k$, we will show that the Betti number of the $k$th dual Whitney homology is a polynomial in $(n-1)$ with nonnegative coefficients. By Theorem~\ref{BjStproof},   to compute the action of $S_n$, it is enough to carry out the appropriate M\"obius number (in effect, Betti number) computations.

For the dual Whitney homology, for $0\le i\le k$ we have 
\[W\!H_{k+1-i}^*(A^*_{n,k})=\bigoplus_{x:|x|=i} \tilde{H}(x,\hat 1)_{A^*_{n,k}}.\]
 Computing Betti numbers, and using Lemma~\ref{UpperIntSubwordOrder}, we have that the dimension of the dual Whitney homology module  equals 
\[\sum_{\stackrel{ x \text{ any word }}{|x|=i}} (-1)^{k+1-i}\mu(x,\hat 1)_{A^*_{n,k}}
=\sum_{\stackrel{ x \text{ any word }}{|x|=i}} {k\choose i} (n-1)^{k-i}=n^i{k\choose i} (n-1)^{k-i}\]
This expression translates into the one in the statement of the proposition, since the trace of $g$ on $S_{(n-1,1)}\oplus S_{(n)}$ is the number of fixed points of $g.$
The second expression is obtained from the binomial expansion of
$n^i$ into powers of $(n-1).$
\end{proof}

\begin{cor} The top homology of $A^*_{n,k}$ as an $S_n$-module is also given by 
 the alternating sums
\[\sum_{i=0}^{k}(-1)^{k-i}\left(S_{(n-1,1)}\oplus S_{(n)}^{\otimes i}\otimes{k\choose i} S_{(n-1,1)}^{\otimes (k-i)}  \right)\]
\[=\sum_{i=0}^k (-1)^{k-i} \bigoplus_{j=0}^i{k\choose i}{i\choose j} S_{(n-1,1)}^{\otimes j+(k-i)},\]
and thus both are equal to $S_{(n-1,1)}^{\otimes k}.$
\end{cor}
\begin{proof} 
The two expressions are simply the alternating sums of dual Whitney homology modules in \eqref{dualWhitney}.  They equal the top homology module by Theorem~\ref{WHmainSS}.
\end{proof}

We can now prove the main result of this section, which generalises Theorem~\ref{BjStproof} 
 to  the rank-set $[r,k]$ consisting of the interval of consecutive ranks $r, r+1, \ldots, k.$ To do this, we must rewrite the partial alternating sums of terms appearing in the dual Whitney homology \eqref{dualWhitney}  as a nonnegative linear combination rather than a signed sum. 
 The poset of words in an alphabet of size $n,$ with lengths bounded above by $k$ and below by $r,$ has homology as follows.

\begin{thm}\label{Bordeaux1991mai} 
%
Fix $k\geq 1$ and let $S$ be the interval of consecutive ranks $[r,k]$ for $1\le r\le k.$ Then the rank-selected subposet $A^*_{n,k}(S)$  has unique nonvanishing homology in degree $k-r,$ and the $S_n$-homology representation on $\tilde{H}_{k-r}(A^*_{n,k}(S))$ 
is given by the decomposition
\begin{equation}\label{epicstruggle}\bigoplus_{i=1+k-r}^k  b_i\,S_{(n-1,1)}^{\otimes i},  \text{ where }b_i={k\choose i}{i-1\choose k-r}, i=1+k-r, \ldots, k.\end{equation}
\end{thm}

\begin{proof}  For brevity we will simply write $\tilde{H}([i,k])$ for the homology of the subposet $A^*_{n,k}(S)$ when $S=[i,k].$  Shellability implies that 
the rank-selected subposet $A^*_{n,k}(S)$  has unique nonvanishing homology in degree $k-i.$  %

Recall again from Theorem~\ref{BjStproof} that it suffices to work with the Betti numbers, for which \eqref{ConsecRanks} in Theorem~\ref{WHmainSS}, in conjunction with Theorem~\ref{WHsubword}, gives the following recurrence for $1\le i\le k-1:$
\begin{equation}\label{ConsecRanksSubword} \mathrm{dim}\, \tilde{H}([i,k])\oplus \mathrm{dim}\, \tilde{H}([i+1,k])=\mathrm{dim}\, W\!H^*_{k+1-i}(A^*_{n,k})=
n^i{k\choose i} (n-1)^{k-i}. \end{equation}
We will prove the Betti number version of (\ref{epicstruggle}) by induction on $i.$ Note that the result is true for $i=1,$ since in that case the formula  in (\ref{epicstruggle}) gives simply $S_{(n-1,1)}^{\otimes k},$ with Betti number $(n-1)^k$, in agreement with Theorem~\ref{Farmer}.  

When $r=k,$ the formula (\ref{epicstruggle}) reduces to $\sum_{i=1}^k {k\choose i}{i-1\choose 0}(n-1)^i,$ which equals 
$n^k-1.$ This is easily seen to be the correct M\"obius number (up to sign) since we then have a single rank consisting of the $n^k$ words of length $k.$ 
Also observe that when $i=k-1,$ the recurrence (\ref{ConsecRanksSubword}) gives 
\[\mathrm{dim}\, \tilde{H}([k-1,k])=n^{k-1} k (n-1) -n^k+1=(k-1)n^k -kn^{k-1}+1.\]

Let $i=1.$ The recurrence \eqref{ConsecRanksSubword} gives 
\[\mathrm{dim}\, \tilde{H}([2,k])=n{k\choose 1} (n-1)^{k-1}-\mathrm{dim}\, \tilde{H}([1,k])\]
\[=kn (n-1)^{k-1}-(n-1)^{k}=(k-1) (n-1)^k+k(n-1)^{k-1},\]
and this coincides with \eqref{epicstruggle} for $r=2.$ 

Assume that \eqref{epicstruggle} holds for the rank-set $S=[r,k].$ We will show that it must hold for $S=[r+1,k].$ 
By hypothesis we have $\mathrm{dim}\, \tilde{H}([r,k])= \sum_{j=1+(k-r)}^k {k\choose j}{j-1\choose k-r} (n-1)^j,$ and hence the recurrence \eqref{ConsecRanksSubword} gives, for $\mathrm{dim}\, \tilde{H}([r+1,k])$, the expression
\[{k\choose r} (n-1)^{k-r} n^r - \sum_{j=1+(k-r)}^k {k\choose j}{j-1\choose k-r} (n-1)^j.\]
Expanding $n^r$ in powers of $(n-1)$, we obtain
\[{k\choose r} (n-1)^{k-r} \sum_{i=0}^r {r\choose i}(n-1)^i - \sum_{j=1+(k-r)}^k {k\choose j}{j-1\choose k-r} (n-1)^j.\]
The coefficient of $(n-1)^{k-r} $ is clearly ${k\choose r}={k\choose k-r}{r\choose 0},$ in agreement with ~(\ref{epicstruggle}).
For $j=1+(k-r), \ldots r+(k-r),$ the term $(n-1)^j$ appears with coefficient $c_j$ where 
\begin{align*}&c_j= {k\choose r} {r\choose j-k+r} -{k\choose j}{j-1\choose k-r}\\
&={k\choose j}\left( \frac{j!(k-j)!}{r!(k-r)!} \frac{r!}{(j-k+r)!(k-j)!} - {j-1\choose k-r}\right)\\
&={k\choose j}\left({j\choose k-r}-{j-1\choose k-r}\right)={k\choose j}{j-1\choose k-r-1},
\end{align*}
which is precisely as predicted by (\ref{epicstruggle}) for $S=[r+1,k].$
This finishes the inductive step, and hence the proof.
\end{proof}

This proof establishes the following combinatorial identity, which will be instrumental in the proof of Theorem~\ref{AlmosthpositiveInstances} later in the paper. 
\begin{cor}\label{NewBinomialIdentity?}\begin{multline*} \sum_{i=0}^{k+1-r} (-1)^i \mathrm{dim}\, W\!H^*_{k+1-(r+i)}(A^*_{n,k})\\
=\sum_{i=0}^{k-r} (-1)^i {k\choose r+i}n^{r+i}(n-1)^{k-(r+i)}+(-1)^{k+1-r}
=\sum_{i=1+k-r}^k {k\choose i}{i-1\choose k-r} (n-1)^i.
\end{multline*}
\end{cor}

\section{Deleting one rank from $A^*_{n,k}$: a curious isomorphism of homology}

In this section we will determine the homology representation of the rank-selected subposet $A^*_{n,k}(S)$ of $A^*_{n,k}$ when $S$ is obtained by deleting one rank from the interval $[1,k].$   In this special case   the computation will reveal a curious duality  in homology. 

Again we use a method developed in \cite{Su0} which is particularly useful for Lefschetz homology computations when the deleted set is an antichain.  The  version  below is the special case when one rank is deleted.
\begin{thm}\label{antichain} (\cite[Theorem 1.10]{Su0}, \cite{SuJer}) Let $P$ be a Cohen-Macaulay poset of rank $r$, $G$ a group of automorphisms of $P$ and let $Q$ be a subposet obtained by deleting a 
 rank-set $T$ consisting of one rank in $P$.  Thus $Q$ is also $G$-invariant, and  $Q$ is graded and has homology concentrated in the highest degree $rank(Q)-2$. Then one has the $G$-equivariant decomposition
\begin{multline}(-1)^{r-rank(Q)}\tilde{H}(Q)-\tilde{H}_{r-2}(P)
=\bigoplus_{x\in T/G}(-1)\cdot(\tilde{H}(\hat 0, x)_P\otimes \tilde{H}(x,\hat 1)_P)\uparrow_{stab(x)}^G,\\
 \text{ where the sum runs over one element } x\in T \text{ in each orbit of }G.
\end{multline}
Here $stab(x)$ denotes the stabiliser subgroup of $G$ which fixes the element $x.$
\end{thm}

We apply this theorem to the poset $A^*_{n,k}$ and the rank-set $S=[1,k]\backslash\{r\},$ removing all words of length $r,$ for a fixed $r$ in $[1,k].$

\begin{thm}\label{rank-deletion} As an $S_n$-module, we have 
\[\tilde{H}_{k-2}(A^*_{n,k}(S))\simeq \left[{k\choose r}-1\right]S_{(n-1,1)}^{\otimes k}
\oplus {k\choose r} S_{(n-1,1)}^{\otimes k-1}.\]
\end{thm}
\begin{proof} 
We invoke Theorem~\ref{BjStproof} by fixing a rank-set $S$ and considering the family of posets $P_n=A^*_{n,S}=A^*_{n,k}(S)$, where $n=|A|$. Once again we need only compute M\"obius numbers in Theorem~\ref{antichain}. Writing simply $\mu(P)$ for the M\"obius number of the poset $P$, the Betti number identity given by the theorem    is 
\[-(-1)^{k-2}\mu(A^*_{n,k}(S))-(-1)^{k-1} \mu(A^*_{n,k})=(-1)\cdot
\sum_{x:|x|=r} (-1)^{r} \mu(\hat 0,x)_{A^*_{n,k}}\cdot (-1)^{k+1-r} \mu(x, \hat 1)_{A^*_{n,k}},\]
or equivalently, clearing signs, 
\[\mu(A^*_{n,k}(S))- \mu(A^*_{n,k})=(-1)\cdot
\sum_{x:|x|=r}  \mu(\hat 0,x)_{A^*_{n,k}}\cdot  \mu(x, \hat 1)_{A^*_{n,k}}.\]
The  summand corresponding to a word $x$ of length $r$ in the right-hand side of this equation is nonzero only if $x$ is a normal word,  by Theorem~\ref{Farmer}.
We therefore obtain,  using Lemma~\ref{UpperIntSubwordOrder},
\begin{multline*}\mu(A^*_{n,k}(S))- \mu(A^*_{n,k})=(-1)\cdot(-1)^r n(n-1)^{r-1}\mu(x_0,\hat 1)_{A^*_{n,k}}\\=(-1)\cdot(-1)^r n(n-1)^{r-1}(-1)^{k-r+1} (n-1)^{k-r} {k\choose r}
=(-1)^k n (n-1)^{k-1}{k\choose r},
\end{multline*}
for any fixed normal word $x_0$ of length $r$.

Hence 
\begin{align*}(-1)^k\mu(A^*_{n,k}(S))&=(-1)^k\mu(A^*_{n,k})+ n (n-1)^{k-1}{k\choose r}
=-(n-1)^k+{k\choose r}n (n-1)^{k-1}\\
&=\left[{k\choose r}-1\right] (n-1)^k + \binom{k}{r} (n-1)^{k-1}
\end{align*}
Since $A^*_{n,k}(S)$ has rank $k,$ this is precisely the Betti number version of the statement of the theorem, thereby completing the proof.
\end{proof}

An immediate and intriguing corollary is the following.
\begin{prop}\label{rank-deletion-duality}  Let $|A|=n.$ Fix a rank $r\in [1,k-1].$ Then the homology modules of the subposets $A^*_{n,k}({[1,k]\backslash\{r\}})$ and 
$A^*_{n,k}({[1,k]\backslash\{k-r\}})$ are $S_n$-isomorphic.
\end{prop}

It would be interesting to explain this isomorphism topologically. More precisely:
\begin{qn} Is there a combinatorial map giving an $S_n$-homotopy equivalence between the simplicial complexes associated to  $A^*_{n,k}({[1,k]\backslash\{r\}})$ and 
$A^*_{n,k}({[1,k]\backslash\{k-r\}})$?  
\end{qn}
  
\section{The action on chains, and arbitrary rank-selected homology} 

 Assume $|A|=n.$ 
  For a subset $S\subseteq [1,k],$ denote by $\alpha_n(S)$ the permutation module of $S_n$ afforded by the maximal chains of the rank-selected subposet $A^*_{n,k}(S).$   In this section we derive a recurrence for the action, and hence an explicit formula.  We begin with an analogue of Theorem~\ref{BjStproof} for the chains.
  
  \begin{prop}\label{BjStproofChains}
  Let $\{P_n\}$ be any sequence of finite posets each carrying an action of the symmetric group $S_n,$ such that for any $g\in S_n,$ the fixed-point subposet $P_n^g$ is isomorphic to the poset $P_{\mathrm {fix}(g)},$  where $\mathrm {fix}(g)$ is the number of fixed points of $g$ as a permutation of $S_n$, %
   Suppose that the  number of maximal chains of $P_n$ is a polynomial in $(n-1),$  say $\sum_{i\geq 0} a_i (n-1)^i. $ 
Then the permutation action of $S_n$ on the maximal chains of $P_n$  decomposes  as a sum of $i$th tensor powers of  the irreducible indexed by the partition $(n-1,1), i\geq 0,$ with coefficient equal to $a_i.$  In particular, the $S_n$-module structure of the maximal chains of $P_n$ is completely determined by its dimension.
  \end{prop}
  \begin{proof}   Since $S_n$ acts by permuting the chains, the trace of $g\in S_n$  on the chains of $P_n$ is equal to the number of chains fixed by $g.$ As in the proof of Theorem~\ref{BjStproof}, the key point is that  this in turn is the number of chains in the fixed-point poset $P_n^g,$ and the latter coincides with $P_{\mathrm {fix}(g)}.$ 
  \end{proof}
  
Note that, as was the case with Theorem~\ref{BjStproof},   Proposition~\ref{BjStproofChains} applies to all rank-selected subposets of $A^*_{n,k}.$  Before we apply this, we state the following reformulation of 
 an observation of Bj\"orner recorded in Theorem~\ref{Bjgf}.    Let $S=\{1\le s_1<s_2<\ldots<s_p\le k\}$ be a subset of  $[1,k].$  By Part (2) of 
  Theorem~\ref{Bjgf}, the number of words in $[\beta, \infty]$ depends only on $|\beta|$.  This immediately gives the following recurrence for the dimensions of the modules $\alpha_n(S):$ 
  \begin{equation}\label{ChainRecDim}\dim \alpha_n(S)=\dim \alpha_n(S\backslash\{s_p\})\sum_{i=0}^{s_p-s_{p-1}}\! {s_p\choose i} (n-1)^i, \! \  \dim \alpha_n(\{s_1\})=n^{s_1}\!=\sum_{i=0}^{s_1} {s_1\choose i} (n-1)^i.  \end{equation}

  \begin{thm}\label{RankSelectedChains} For any subset $S\subseteq [1,k],$ the $S_n$-module induced by the action of $S_n$ on the maximal chains of the rank-selected subposet $A^*_{n,k}(S)$ is a \textbf{nonnegative} integer combination of tensor powers of the irreducible indexed by $(n-1,1).$ Hence the $S_n$-representation on the homology of the rank-selected subposet $A^*_{n,k}(T), T\ne \emptyset,$ is an integer combination of positive tensor powers of the irreducible indexed by $(n-1,1).$ The highest tensor power that can occur is the $m$th, where $m=\max(T).$ 
  \end{thm}
  \begin{proof}
  Let $S=\{1\le s_1<s_2<\ldots<s_p\le k\}.$ From Proposition~\ref{BjStproofChains}, it suffices to compute the dimension of the module of maximal chains in $A^*_{n,k}(S)$ as a polynomial in $(n-1).$ 
  Hence \eqref{ChainRecDim} immediately gives the following recursive description for the modules $\alpha_n(S).$ 
  \begin{equation}\label{ChainRec}\alpha_n(S)=\alpha_n(S\backslash\{s_p\})\otimes  \bigoplus_{i=0}^{s_p-s_{p-1}} {s_p\choose i}S_{(n-1,1)}^{\otimes i}, \end{equation}
  and \begin{equation}\label{ChainRecSingle}\alpha_n(\{s_1\})=\bigoplus_{i=0}^{s_1} {s_1\choose i} S_{(n-1,1)}^{\otimes i}.\end{equation}
\noindent
  By induction  it is clear that $\alpha_n(S)$ is a nonnegative integer combination of $S_{(n-1,1)}^{\otimes j}, 0\le j\le m=\max(S).$
   It is also clear that the 0th tensor power, that is, the trivial module $S_{(n)}$, occurs exactly once in each $\alpha_n(S).$ 
  
  Note that when $S=\emptyset,$ the homology is simply the trivial module.
  The claim about the decomposition of the homology into tensor powers of $S_{(n-1,1)}$ now follows from 
   Stanley's  theory of rank-selected homology representations \cite{St1}. We have
  \begin{equation}\label{RPSbeta}
  \alpha_n(T) =\sum_{S\subseteq T} \beta_n(S) \quad\text{  and thus  }\quad\beta_n(T)=\sum_{S\subseteq T} (-1)^{|T|-|S|} \alpha_n(S),\end{equation}
  where $\beta_n(S)$ is the representation of $S_n$ on the homology of the rank-selected subposet $A^*_{n,k}(S)$ of $A^*_{n,k}.$  When $T$ is nonempty, it is clear from the previous paragraph that the occurrences of the 0th tensor power, which equals $S_{(n)},$ all cancel in ~\eqref{RPSbeta}; the trivial module occurs with coefficient $\sum_{S\subseteq T} (-1)^{|T|-|S|},$ which is zero.  Hence only  positive tensor powers will appear.  
   \end{proof}
  
  Thus Theorem~\ref{RankSelectedChains}   supports Conjecture~\ref{Conj:tensorpower}.
  Note that it is easy to concoct signed integer combinations of tensor powers that are not true $S_n$-modules. For instance,  the integer combination $S_{(n-1,1)}^{\otimes 2}-2\, S_{(n-1,1)}$ decomposes into $S_{(n)}+S_{(n-2,2)}S_{(n-2,1,1)}-S_{(n-1,1)},$ while $S_{(n-1,1)}^{\otimes 2}-\, S_{(n-1,1)}=S_{(n)}+S_{(n-2,2)}+S_{(n-2,1,1)}$ is a true $S_n$-module.  Also see Theorem~\ref{myBrauer} later in the paper.

  It is worth pointing out the special case  for the full poset $A^*_{n,k}.$
  
   \begin{thm}\label{Chains} The action of $S_n$ on the maximal chains of ${A^*_{n,k}}$ decomposes into the direct sum of tensor powers 
  \[ S_{(n)}\oplus \bigoplus_{j=1}^k c(k+1,j) S_{(n-1,1)}^{\otimes k+1-j},\]
  where $c(k+1,j)$ is the number of permutations in $S_{k+1}$ with exactly $j$ cycles in its disjoint cycle decomposition.
  \end{thm}
  
  \begin{proof}   Specialising \eqref{ChainRecDim} to the case $S=[1,k]$ gives the recurrence 
  $\dim\alpha_n([1,k]) =\dim\alpha_n([1,k-1]) (1+k(n-1)),$  and clearly $\dim\alpha_n([1,1])=n.$
  It follows that 
  \[\dim\alpha_n([1,k])=\prod_{i=1}^{k} (1+i(n-1)),\]
  a formula  due to Viennot \cite[Lemma 4.1, Proposition 4.2]{V}.
  
  Using the generating function (see \cite{St3EC1})  $\sum_{j=1}^m c(m,j)t^j = t(t+1)(t+2)\ldots (t+(m-1))$,  we find that 
  \[\dim\alpha_n([1,k])=1+\sum_{j=1}^k c(k+1,j) (n-1)^{k+1-j}.\]
  Invoking Proposition~\ref{BjStproofChains}, the result follows, noting that the constant term in the above expression corresponds to the occurrence of the trivial representation. \end{proof}

  By expanding the expression for $\alpha_n(S)$ in ~\eqref{ChainRec}, we have the following  observation. Although $S_{(n-1,1)}$ is a quotient of two permutation modules, it is not clear how to deduce this corollary directly.  Later in the paper we will examine these tensor powers more carefully; see Lemma~\ref{RefRephPos}.
  
  \begin{cor} The $S_n$-module $S_{(n)}\oplus \bigoplus_{j=1}^k c(k+1,j) S_{(n-1,1)}^{k+1-j}$ is in fact a permutation module.  
  More generally, for any subset $S=\{1\le s_1<\ldots<s_p\le k\}$ of $[1,k],$ the $S_n$-module 
  \begin{equation}\label{tensorchain}\bigotimes_{r=1}^p\left(\bigoplus_{i=0}^{s_r-s_{r-1}}{s_r\choose i} S_{(n-1,1)}^{\otimes i}\right), s_0=1,\end{equation}
  is a permutation module.
  \end{cor}
  
  \begin{proof} The expression ~\eqref{tensorchain} gives the $S_n$-action on the chains of the rank-selected subposet 
  $A^*_{n,k}(S)$ and is therefore a permutation module. \end{proof} 
  
    By applying Proposition~\ref{BjStproofChains}, we have the following two descriptions of the action on chains between two ranks. 
  \begin{prop}\label{2chains} For $S=\{1\le s_1<s_2\le k\},$ 
  $\alpha_n(\{s_1<s_2\})$ is given by 
  \begin{enumerate}
  
  \item \[\left(\bigoplus_{j=0}^{s_1} \binom{s_1}{j} 
  S_{(n-1,1)}^{\otimes j}\right)
  \bigotimes \left( \bigoplus_{i=0}^{s_2-s_1} \binom{s_2}{i} S_{(n-1,1)}^{\otimes i}\right)\]
  \item %
and also by  \[ (S_{(n-1,1)}\oplus S_{(n)})^{\otimes s_1}
  \bigotimes \bigoplus_{j=0}^{s_2-s_1}  \left[\binom{s_1+j-1}{j}S_{(n-1,1)}^{\otimes j}\bigotimes
  (S_{(n-1,1)}\oplus S_{(n)})^{\otimes s_2-s_1-j}\right].\]
  \end{enumerate}
  \end{prop}
  \begin{proof} From Proposition~\ref{BjStproofChains}, it suffices to compute the dimension of the module of chains as a polynomial in $(n-1).$ From  Eqn.~(\ref{ChainRecDim}),  the dimension   of $\alpha_n(\{s_1<s_2\})$ is given by 
  \[n^{s_1}\sum_{i=0}^{s_2-s_1} \binom{s_2}{i} (n-1)^i.\qquad (A)\]
  Recall  Bj\"orner's generating function for the zeta function of subword order, Part (3) of Theorem~\ref{Bjgf}, which we now use to count the number of chains from a fixed element to all elements above it of a fixed rank.  This  gives (by extracting the coefficient of $t^{s_2}$ in the right-hand side of Part (3)):
  \[n^{s_1}\sum_{i,j\ge0, i+j=s_2-s_1}  \binom{s_1+j-1}{j} (n-1)^j n^i=
  n^{s_1}\sum_{j=0}^{s_2-s_1}  \binom{s_1+j-1}{j} (n-1)^j n^{s_2-s_1-j}.\qquad (B)\]
 Since as usual $(n-1)$ is the dimension of $S_{(n-1,1)}$ and $n$ is the dimension of $S_{(n)}\oplus S_{(n-1,1)},$ invoking Proposition~\ref{BjStproofChains},  we conclude  that $(A)$ and $(B)$  correspond respectively to  the $S_n$-module decompositions in Part (1) and Part (2). \end{proof}
 
\begin{rk} 
  By expanding in powers of $(n-1)$, the equivalence of the two expressions for the dimension of $\alpha_n(\{s_1<s_2\})$ is equivalent to the following  binomial coefficient identity:

\begin{center}$  \binom{a+r}{i}=\sum_{j=0}^i \binom{r-j}{i-j}\binom{a+j-1}{j} \text{ for } 0\le i\le r, $\end{center}
\noindent
  or equivalently, putting $k=r-i,$ 

\begin{center} $ \binom{a+r}{a+k}=\sum_{j=0}^{r-k} \binom{r-j}{k}\binom{a+j-1}{a-1}
   \text{ for } 0\le k\le r.$\end{center}
\end{rk}
  
 We can now show that Conjecture~\ref{Conj:tensorpower} is true for rank sets of size 2.
 
  \begin{thm}\label{TwoRanks}  Let $S=\{1\le s_1<s_2\le k\}$ be a rank-set of size 2 in $A^*_{n,k}.$ The homology representation of $A^*_{n,k}(S)$ is given by 
  \[\sum_{v=1}^{s_1} S_{(n-1,1)}^{\otimes v} (c_v -\binom{s_1}{v}) \bigoplus \sum_{v=1+s_1}^{s_2} S_{(n-1,1)}^{\otimes v} c_v,\]
  where $c_v$ is the following positive integer:
  \[c_v
  = \sum_{j=1}^{\min(v, s_2-s_1)}\binom{s_2-j}{v-j}\binom{s_1+j-1}{j} .\]
  Moreover $c_v\ge \binom{s_1}{v}$ when $v\le s_1,$ and hence 
  the homology is a nonnegative integer combination of positive tensor powers of $S_{(n-1,1)}.$
  \end{thm}
  
  \begin{proof} In order to establish the positivity, it is (curiously) easier to work with the second formulation of Proposition~\ref{2chains}.  By Theorem~\ref{BjStproof}, it is enough to show that the dimension of the homology module is a polynomial in $(n-1)$ with nonnegative integer coefficients.
   We have \[\beta_n(\{s_1<s_2\})=\alpha_n(\{s_1<s_2\})-\alpha_n(\{s_1\})-\alpha_n(\{s_2\})+\alpha_n(\emptyset)\]
   which yields, (in terms of dimensions, from $(B)$ in the proof of Proposition~\ref{2chains})
   \begin{align*} 
   & n^{s_1} \sum_{j=0}^{s_2-s_1} n^{s_2-s_1-j} (n-1)^j \binom{s_1+j-1}{j} -n^{s_2}-n^{s_1}+1\\
   &=\sum_{j=1}^{s_2-s_1} n^{s_2-j} (n-1)^j 
   \binom{s_1+j-1}{j} -(n^{s_1}-1).
   \end{align*}
  Expanding $n^{s_1}$ and $n^{s_2-j}$ in nonnegative powers of $(n-1)$ gives
  $\beta_n(\{s_1<s_2\})=$
  \begin{align} \sum_{j=1}^{s_2-s_1}\sum_{u=0}^{s_2-j} \binom{s_2-j}{u} (n-1)^{u+j}\binom{s_1+j-1}{j}
  -\sum_{j=1}^{s_1} \binom{s_1}{j} (n-1)^j\\
  =\sum_{v=1}^{s_2} (n-1)^v c_v -\sum_{j=1}^{s_1} \binom{s_1}{j} (n-1)^j,  
  \end{align}
  where \[c_v=\sum_{\stackrel{(u,j): u+j=v}{ 1\le j\le s_2-s_1,\, 0\le u\le s_2-j}}
  \binom{s_2-j}{v-j}\binom{s_1+j-1}{j}.\] 
  The latter sum runs over all $j$ such $1\le j\le s_2-s_1$ and $0\le v-j\le s_2-j,$ i.e. over all $j=1,\ldots, \min(v, s_2-s_1),$ as stated.
  
  Now $c_v$ is a sum of nonnegative integers for each $v=1,\ldots, s_2.$ When $v\le s_1,$ the $j=1$ summand of $c_v$ can be seen to be $\binom{s_2-1}{v-1} s_1$, and so 
  \begin{align*}  c_v-\binom{s_1}{v}\ge \binom{s_2-1}{v-1} s_1-\binom{s_1}{v}&=
  \dfrac{s_1!}{v!(s_2-v)!} (v \dfrac{(s_2-1)!}{(s_1-1)!}-\dfrac{(s_2-v)!}{(s_1-v)!})\\
  &=\dfrac{s_1! (s_2-s_1)!}{v!(s_2-v)!}( v\binom{s_2-1}{s_2-s_1}
  -\binom{s_2-v}{s_2-s_1})
  \end{align*}
  and this is clearly nonnegative, since  $\binom{s_2-1}{s_2-s_1}
  \ge \binom{s_2-v}{s_2-s_1}$ for $v\ge 1.$ We have shown that the Betti number of $\beta_n(\{s_1<s_2\})$ is a nonnegative integer combination of  positive powers of $(n-1),$ as claimed.
  \end{proof}
  
\begin{rk}\label{Ranks1anda}  Let $s_1=1$, and consider the two ranks $\{1<s_2\}.$
 The homology of the rank-selected subposet is then 
 \[\bigoplus_{v=2}^{s_2-1}\binom{s_2}{v-1}S_{(n-1,1)}^{\otimes v} \bigoplus 
(s_2-1) S_{(n-1,1)}^{\otimes s_2}.\]
 \end{rk}

\section{Tensor powers of the reflection  representation I}\label{PowersOfRefl}

In this section we explore the tensor powers $S_{(n-1,1)}^{\otimes k}.$  The paper 
\cite{CG} gives a combinatorial model for determining the multiplicity of an irreducible in the $k$th tensor power, and an explicit formula in the case when $n$ is sufficiently larger than $k.$  We give general formulas that apply to the case of arbitrary tensor powers.

  We use symmetric functions to describe some of the results that follow.  The homogeneous symmetric function $h_n$ is the Frobenius characteristic, denoted $\mathrm{ch},$ of the trivial representation of $S_n.$ Also let $*$ denote the internal product on the ring of symmetric functions, so that the Frobenius characteristic of the  Kronecker product of two $S_n$-modules is the internal product of the two characteristics. See \cite[p.115]{M} and \cite[Chapter 7, p. 476]{St4EC2}.
Recall that the natural representation of $S_n$ is the permutation action on a set of $n$ objects.  The stabiliser of any one object is 
the Young subgroup $S_1\times S_{n-1},$ and hence the natural representation is given by the induced module $1\uparrow_{S_1\times S_{n-1}}^{S_n}$, with Frobenius characteristic $h_1h_{n-1}$. In particular we have the decomposition  %
\[1\uparrow_{S_1\times S_{n-1}}^{S_n}=S_{(n)}\oplus S_{(n-1,1)}=S_{(n,1)}\downarrow_{S_n}^{S_{n+1}}.\]
The following lemma is an easy exercise in permutation actions. We sketch a proof for completeness.
\begin{lem}\label{Stirling} Let $V_{j,n}$ denote the permutation module obtained from the $S_n$-action on the cosets of the Young subgroup $S_1^j\times S_{n-j}.$ Then the $k$th tensor power of the natural representation $V_{1,n}$ of $S_n$ decomposes into a sum of $S(k,j)$ copies of $V_{j,n},$ where $S(k,j)$ is the Stirling number of the second kind:
\begin{equation}\label{StirlingNat} V_{1,n}^{\otimes k}
=\sum_{j=1}^{\min(n,k)}S(k,j)\, V_{j,n}, \quad \text{ and thus} \quad
(h_1h_{n-1})^{* k}=\sum_{j=1}^{\min(n,k)} S(k,j) \, h_1^j h_{n-j}.
\end{equation}
\end{lem}
\begin{proof} If the module $V_{1,n}$ is realised as $\mathbb{C}^n$ with basis $\{v_1\ldots,v_n\}$, say, then $V_{1,n}^{\otimes k}$ is realised by 
the $k$th tensor power of $\mathbb{C}^n,$ with $n^k$ basis elements 
$v_{i_1}\otimes \ldots \otimes v_{i_k}, \quad 1\le i_1, \ldots, i_k\le n.$
The $S_n$-action now permutes these $n^k$ basis elements.  To determine the orbits, note that 
there is a surjection from this basis of tensors to the set partitions of a $k$-element set into nonempty blocks.  Each such partition with $j$ blocks indexes  an orbit of the $S_n$-action, with stabiliser (conjugate to) $S_1^j\times S_{n-j}$. The blocks correspond to repetitions of a $v_i$ in the tensor; thus $a,b$ belong in the same block if $v_a=v_b$ in the tensor. The orbit is the transitive permutation representation with Frobenius characteristic $h_1^j h_{n-j}.$

The last statement is now immediate.  %
\end{proof}

\begin{ex} We illustrate the above argument  with an example.  With $n=5$ and $k=7,$ the tensor 
$v_5\otimes v_2\otimes v_2\otimes v_4\otimes v_2\otimes  v_4\otimes v_5$ maps to the partition $17-235-46$ of a set of size $k=7$ into $j=3$ blocks, corresponding to the three distinct basis elements $v_2, v_4, v_5$ of $\mathbb{C}^n$.  Its orbit under $S_5$ consists of all basis tensors 
$v_{i_1}\otimes \ldots \otimes v_{i_7}$ such that $v_{i_1}=v_{i_7}, v_{i_2}=v_{i_3}=v_{i_5},$ and $v_{i_4}=v_{i_6}.$
Writing  $S_A$ for the permutations of the elements of $A,$ for any subset $A$ of positive integers, the stabiliser is  $S_{\{5\}}\times S_{\{2\}} \times S_{\{4\}} \times S_{\{1,3\}},$ conjugate to the Young subgroup indexed by the integer partition $(2,1,1,1)$ of 5.  
\end{ex}

\begin{rk}\label{IndRes}  This lemma can also be proved by iterating a standard representation theory result, namely that for finite groups $G$ and $H$ with $H$ a subgroup of $G,$ and $G$-module $W,$ $H$-module $V,$ $W\otimes (V\uparrow_H^G)=
(W\downarrow_H\otimes V)\uparrow_H^G.$  In our case $G=S_n$ and $H$ is the Young subgroup $S_1\times S_{n-1}.$ %
\end{rk}
\begin{thm}\label{StirlingHom} The top homology of $A^*_{n,k}$ has Frobenius characteristic 
\[\sum_{i=0}^{\min(n,k)} h_1^i h_{n-i} \left( \sum_{r=0}^{k-i} 
(-1)^{r}\binom{k}{r}S(k-r,i)\right).\]
\end{thm}
\begin{proof} Observe that the  Frobenius characteristic of 
$(S_{(n-1,1)})^{\otimes k}$  is the $k$-fold internal product of $(h_1 h_{n-1}-h_n).$ Standard properties of the tensor product make $*$      a commutative and associative product in the ring of symmetric functions, so we have
\begin{align*} &(h_1 h_{n-1}-h_n)^{* k} \\
&=\sum_{j=0}^k \binom{k}{j} (-1)^{k-j}(h_1 h_{n-1})^{* j} * (h_n)^{* (k-j)}=(-1)^k h_n+\sum_{j=1}^k \binom{k}{j} (-1)^{k-j}(h_1h_{n-1})^{* j}\\
&= (-1)^k h_n+\sum_{j=1}^k \binom{k}{j} (-1)^{k-j}
\sum_{i=1}^{\min(n,j)} S(j,i)  h_1^i h_{n-i}\quad 
 \text{from Lemma~\ref{Stirling}}
\end{align*}
\begin{align*}
&= \sum_{j=0}^k \binom{k}{j} (-1)^{k-j}
\sum_{i=0}^{\min(n,j)} S(j,i)  h_1^i h_{n-i}
=\sum_{i=0}^{\min(n,k)} h_1^i h_{n-i} \left( \sum_{j=i}^{k} 
(-1)^{k-j}\binom{k}{j}S(j,i)\right).
\end{align*}
Note that $S(0,0)=1$ and $S(j,0)=0$ for $j\ge 1.$ Putting $r=k-j$ in the last step gives the result.
\end{proof}

Theorem~\ref{ReflRepMain}  gives a different description of this module, from which it will be evident that the coefficients of $h_1^ih_{n-i}$ are positive for $i\ge 2.$

We can now determine the multiplicity of the trivial representation in 
the top homology of $A^*_{n,k}$:
\begin{cor}\label{TrivRep}  Let $n\ge 2.$ The following are equal:
\begin{enumerate}
\item the multiplicity of the trivial representation in $S_{(n-1,1)}^{\otimes k};$
\item the multiplicity of the irreducible $S_{(n-1,1)}$  in 
$S_{(n-1,1)}^{\otimes k-1};$
\item the number 
\[\sum_{r=0}^k (-1)^r \binom{k}{r} \sum_{i=0}^{\min(n,k)} S(k-r,i).\]
\end{enumerate}
When $n\ge k,$ this multiplicity %
equals the number of set partitions of $\{1,\ldots,k\}$ with no singleton blocks.
\end{cor}

\begin{proof} The first two multiplicities are equal by standard properties of the tensor product, since 
\[ \langle V\otimes W, S_{(n)}\rangle=\langle V, 
S_{(n)}\otimes W\rangle =\langle V,W\rangle \]
The equivalence with the third formula follows from Theorem~\ref{StirlingHom}, since $\langle h_1^i h_{n-i}, h_n\rangle=1$ for all $i$ (alternatively, $1\uparrow_{S_1^i\times S_{n-i}}^{S_n}$ is a transitive permutation module). 
Let $B_n^{\ge 2}$ denote the number of set partitions of $[n]=\{1,\ldots,n\}$ with no  blocks of size 1, and let $B_n$ denote the $n$th Bell number, that is, the total number of set partitions of $[n].$  Inclusion-exclusion  shows that 
\begin{equation}\label{BellNosingleton}B_n^{\ge 2}=\sum_{r=0}^n (-1)^r \binom{n}{r} B_{n-r},
\end{equation}
since the number of partitions containing a fixed set of $r$ singleton blocks is 
$B_{k-r}.$  
When $n\ge k,$ the formula in Part (3) simplifies to 
\[\sum_{r=0}^k (-1)^r \binom{k}{r} \sum_{i=0}^{k-r} S(k-r,i)
=\sum_{r=0}^k (-1)^r \binom{k}{r} B_{k-r}.\]
That this number is $B^{\ge 2}_k,$ the number of partitions of $[k]$ with no singleton blocks, now  follows from Eqn.~(\ref{BellNosingleton}).  (This is sequence  A000296 in OEIS.)
\end{proof}

Corollary~\ref{TrivRepStableCase} in the next section will give a different expression for the multiplicity of the trivial representation, for arbitrary $n,k,$ as a sum of positive integers.

\section{\lq\lq Almost" an $h$-positive permutation module}

We begin by recalling basic facts about permutation modules.  $V$ is a permutation module for $S_n$ if there exists a basis for $V$ that is permuted by the $S_n$-action.   In particular the character values of a permutation module are all nonnegative. For example, the character $\chi_{(n-1,1)}$ of the reflection representation $S_{(n-1,1)}$ has values 
$\chi_{(n-1,1)}(g) = \textrm{fix}(g)-1$ for each permutation $g\in S_n,$ so it is negative for permutations without fixed points, and the same is true for odd tensor powers of $S_{(n-1,1)}.$ Thus odd tensor powers of $S_{(n-1,1)}$ are not permutation modules. In particular, the homology of $A_{n,k}$ itself need not be a permutation module.

A special case of a permutation module occurs when its Frobenius characteristic is $h$-positive, that is, the coefficients in the basis of homogeneous symmetric functions are nonnegative. It is  well known that the homogeneous symmetric function $h_\lambda$ is the Frobenius characteristic of the transitive permutation representation whose orbit stabiliser is the Young subgroup of $S_n$ indexed by $\lambda.$ Hence the $h$-positivity of the Frobenius characteristic implies that it is a permutation module (whose point stabilisers are Young subgroups), but not conversely.  A nice example is provided by the set partitions of $[4]$ into two blocks of size 2, viz.  $12-34,$ $13-24,$ and $14-23$.  The action of $S_4$ is a transitive permutation module whose point stabiliser is the wreath product  $S_2[S_2],$ but its Frobenius characteristic is $h_4+h_2^2-h_1h_3, $ not $h$-positive.

The goal of this section is to prove the following theorem.

\begin{thm}\label{Homologyhbasis}  Let $T\subseteq [1,k]$  be any nonempty subset of ranks in $A^*_{n,k}.$ The following statements hold for the Frobenius characteristic $F_{n,k}(T)$ of the homology representation $\tilde{H}(A^*_{n,k}(T)):$ 
 \begin{enumerate}  
 \item its expansion in the basis of homogeneous symmetric functions is an integer combination supported on the set $T_{1}(n)=\{h_\lambda: \lambda=(n-r, 1^r), r\ge 1\}.$
 \item $F_{n,k}(T) +(-1)^{|T|} s_{(n-1,1)}$ is supported on the set 
 $T_{2}(n)=\{h_\lambda: \lambda=(n-r, 1^r), r\ge 2\}.$
 \end{enumerate}
 \end{thm}

When $F_{n,k}(T)+(-1)^{|T|} s_{(n-1,1)}$ is in fact a nonnegative integer combination of $T_{2}(n)=\{h_\lambda: \lambda=(n-r, 1^r), r\ge 2\},$ 
we may view $F_{n,k}(T)$ as being almost a permutation module, hence the title of this section.  First we prove a stronger result for the action on the chains. 

\begin{thm}\label{Chainhbasis}  Let $S\subseteq[1,k].$ If $|S|\ge 1,$  the $S_n$-module induced by the action of $S_n$ on the space of chains $\alpha_n(S)$ has $h$-positive Frobenius characteristic supported on the set $T_{1}(n)=\{h_\lambda: \lambda=(n-r, 1^r), r\ge 1\}.$ %
 Furthermore, $h_1 h_{n-1}$ always appears with coefficient 1 in the $h$-expansion of $\alpha_n(S)$.
 \end{thm}
 
 \begin{proof}    Recall that $*$ denotes the inner tensor product. 
Note that the case of a single rank has already been established in Lemma~\ref{Stirling}: if $S=\{s_1\}$, then 
 $\mathrm{ch} \,\alpha_n(S)= (h_1 h_{n-1})^{*{s_1}}$ since the dimension of the module is $n^{s_1}$, and the coefficient of $h_1 h_{n-1}$ in the $h$-expansion is the Stirling number $S(s_1,1)=1.$
 
 We proceed by induction, using the decomposition \eqref{ChainRec} of Theorem~\ref{RankSelectedChains}.  We have, with $|S|\ge 2,$  
 \begin{equation*}\alpha_n(S)\!=\alpha_n(S\backslash\{s_p\})\otimes  \sum_{i=0}^{s_p-s_{p-1}} {s_p\choose i} S_{(n-1,1)}^{\otimes i}
 =\alpha_n(S\backslash\{s_p\})\otimes \left(\sum_{i=1}^{s_p-s_{p-1}} {s_p\choose i} S_{(n-1,1)}^{\otimes i} \oplus S_{(n)}\right).   \end{equation*}

 Translating Remark~\ref{IndRes}  into Frobenius characteristics gives the well-known symmetric function formula $(h_1 h_{n-1})*f= 
 h_1 \tfrac{\partial}{\partial p_1} f$ for any symmetric function $f$ of homogeneous degree $n$
 \cite[Exercise 7.81, p. 477]{St4EC2}, \cite[Example 3 (c), p. 75]{M}. It is easy to check that 
 for $a\ge 1, a+b=n,$ 
 \begin{equation} \label{IndResIterate} 
 (h_1^ah_b)*s_{(n-1,1)} =\begin{cases}(a-1) h_1^{a} h_b+h_1^{a+1}h_{b-1},\ a\ge 2;\\
h_1^{a+1}h_{b-1},\ a=1.
\end{cases}
 \end{equation}
 
 Note the factor $h_1^2$ when $a\ne 0.$  
 Let $T_{ 2}(n)=\{h_\lambda: \lambda=(n-r, 1^r), r\ge 2\}.$ Iterating ~\eqref{IndResIterate} gives the fact that  when $\mu\in T_{ 1}(n-1),$ 
$ (h_1h_\mu)*s_{(n-1,1)}^{* k}$
 is a nonnegative integer combination of terms in $T_{ 2}(n). $
 
 Assume $S=\{1\le s_1<\ldots<s_p\le k\},$ and  $|S|\ge 2.$  
 In terms of symmetric functions, the decomposition \eqref{ChainRec} of Theorem~\ref{RankSelectedChains} becomes  
 \begin{multline}\mathrm{ch}\, \alpha_n(S)=\mathrm{ch}\, \alpha_n(S\backslash\{s_p\}) 
* (h_n+\sum_{i=1}^{s_p-s_{p-1}} {s_p\choose i} s_{(n-1,1)}^{* i})\\
=\mathrm{ch}\, \alpha_n(S\backslash\{s_p\}) 
+ \mathrm{ch}\, \alpha_n(S\backslash\{s_p\}) 
* \left(\sum_{i=1}^{s_p-s_{p-1}} {s_p\choose i} s_{(n-1,1)}^{* i}\right).\end{multline}
 Suppose now that the first term above, $\mathrm{ch}\,\alpha_n(S\backslash\{s_p\}),$ is a nonnegative integer combination of terms in $T_{ 1}(n),$  in which $h_1h_{n-1}$ appears with coefficient 1. By ~\eqref{IndResIterate}, the $h$-expansion of the second term contains only terms in $T_2(n);$ the crucial point here is that, since $p\ge 2,$ the $h$-expansion of $\mathrm{ch}\,\alpha_n(S\backslash\{s_p\})$ does not contain the function  $h_n.$ It follows that 
 $\mathrm{ch}\,\alpha_n(S)$ must be a nonnegative integer combination of terms in $T_{ 1}(n),$ and the coefficient of $h_1h_{n-1}$ is inherited from $\mathrm{ch}\,\alpha_n(S\backslash\{s_p\}).$  It is therefore equal to 1.  This completes the induction.
 \end{proof}
 
 \noindent
 \textbf{Proof of Theorem~\ref{Homologyhbasis}:}
 \begin{proof}  Using  Stanley's equation for rank-selected homology, Equation~\eqref{RPSbeta} in Theorem~\ref{RankSelectedChains}, we have 
  \[F_{n,k}(T)=\sum_{S\subseteq T} (-1)^{|T|-|S|} \mathrm{ch}\, \alpha_n(S).\]
 
 From Theorem~\ref{Chainhbasis}, this has an expansion in the $h$-basis in which $h_n$ appears only in $\alpha_n(\emptyset)$ with coefficient 1, and $h_1 h_{n-1}$ appears in $\alpha_n(S)$ with coefficient 1 for all nonempty $S.$ Hence the coefficient of $h_n$ in the right-hand side above is $(-1)^{|T|},$ while the coefficient of $h_1 h_{n-1}$ is 
 \[\sum_{S\subseteq T, S\ne \emptyset} (-1)^{|T|-|S|} =
(-1)^{|T|}[ \sum_{i=0}^{|T|} \binom{|T|}{i} (-1)-1]  =(-1)^{|T|-1}.\]
But $(-1)^{|T|}s_{(n-1,1)}=(-1)^{|T|} h_1 h_{n-1} -(-1)^{|T|} h_n,$ and the conclusion follows. \end{proof} 

The preceding theorem motivates Conjecture~\ref{Conj:almost-h-pos} in the Introduction.  We will show that Conjecture~\ref{Conj:almost-h-pos} is true in the following cases of rank-selection:

 \begin{thm}\label{AlmosthpositiveInstances} For any nonempty rank set $T\subseteq[1,k],$ consider the  module 
  $V_T=\tilde{H}_{k-2}(A^*_{n,k}(T))+ (-1)^{|T|} S_{(n-1,1)}.$  In each of the following cases, $V_T$  is a nonnegative combination of transitive permutation modules with orbit stabilisers 
of the form $S_1^d\times S_{n-d}, d\ge 2.$  Equivalently, the symmetric function $F_{n,k}(T) +(-1)^{|T|} s_{(n-1,1)}$ is supported on the set 
 $T_{2}(n)=\{h_\lambda: \lambda=(n-r, 1^r), r\ge 2\}$ with nonnegative integer coefficients in each of the following cases:
\begin{enumerate}
\item $T=[r,k], k\ge r\ge 1.$ 
\item $T=[1,k]\backslash\{r\}, k\ge r\ge 1.$ 
\item $T=\{1\le s_1<s_2\le k\}.$
\end{enumerate}
  \end{thm}
  
  The proof relies on the homology computations of the preceding sections, but we also need to develop additional tools.  In particular Theorem~\ref{ReflRepMain} will be crucial to  the proof.
 We begin by deriving a different expression for the Whitney homology modules of $A^*_{n,k}$, thereby obtaining a new expression for the top homology module as well.
\begin{prop}\label{LowerIntTopo} Let $\alpha$ be a nonempty word in $A^*_{n,k}.$ Then
\begin{enumerate} 
\item If $\alpha$ is not a normal word, the (order complex of the) interval $(\hat 0, \alpha)$ is contractible and hence its homology vanishes in all degrees;
\item If $\alpha$ is  a normal word, the (order complex of the) interval $(\hat 0, \alpha)$ is homotopy equivalent to a single sphere in the top dimension,  and the stabiliser subgroup of $\alpha$ acts trivially on the homology.
\end{enumerate}
\end{prop}
\begin{proof} The topological conclusions in both parts are immediate from the formula for the M\"obius number in Theorem~\ref{Farmer} and  Bj\"orner's dual CL-shellability result of Theorem~\ref{BjCLshellable}. 

If $\alpha$ is normal, then it consists of some $m\le n$ distinct letters $\{x_1,\ldots, x_m\}$, and consecutive letters are distinct.  The stabiliser is the subgroup which fixes each $x_i$ pointwise, and permutes the remaining $n-m$ letters arbitrarily.  It is thus  a product of $m$ copies of the trivial group $S_1$ and the group $S_{n-m}.$  Clearly  this subgroup fixes every element in the interval $(\hat 0, \alpha)$ pointwise, and hence the action on the unique nonvanishing homology is trivial.
\end{proof}

   Let $S^*(j,d)$ denote the number of set partitions of $[j]$ into $d$ blocks, with the property that no block  contains consecutive integers (a  reduced Stirling number).  The surjection in the following lemma will be needed in what follows.
  \begin{lem}\label{surj-normal-to-ptns}  There is a surjection $\psi$ from the set of words of length $j$ in an alphabet of size $n$ to the set partitions into $d$ blocks of $[j],$ 
  where $d$ is the number of distinct letters in $\alpha.$
  This surjection maps normal words  onto  set partitions with the property that no two consecutive integers are in the same block. 
  In particular, the number $n(n-1)^{j-1}$ of normal words of length $j\ge 2$ on an alphabet of size $n$ is also equal to
  \[ \sum_{d= 1}^{\min(n,j)} \dfrac{n!}{(n-d)!} S^*(j,d).\]
  \end{lem}
  \begin{proof} %
  The idea of the proof is similar to that of Lemma~\ref{Stirling}. A word $\alpha$ of length $j$ with $d$ distinct letters maps to a set partition $\psi(\alpha)$ of $[j]$ with 
  $d$ blocks, where positions corresponding to integers in the same block have equal letters in $\alpha.$
  For instance, the word 
  $\alpha= abbcbca,$ of length 7 with 3 distinct letters. Then the set partition of $[7]$ associated to $\alpha$ is $\psi(\alpha)=17-235-46.$ 
  
  Now suppose $\alpha$ is normal.  Then no two consecutive positions have equal letters, which is precisely the condition that no block of 
  the set partition $\psi(\alpha)$ contains consecutive integers.
  
  The last statement is verified by observing that a normal word in the pre-image of every set partition of $[j]$ with $d$ blocks contains 
  $d$ distinct letters chosen out of $n$, which can be permuted amongst themselves in $d$ ways.
  \end{proof}
  
  Recall that the ordinary Stirling numbers of the second kind satisfy the recurrence $S(n+1,d)=S(n,d-1)+d S(n,d)$ with initial conditions $S(0,0)=1$ and $S(n,0)=0=S(0,d)$ if $n,d>0.$ It is easy to verify similarly that the reduced Stirling numbers $S^*(j,d)$ satisfy the recurrence 
  $ S^*(n+1,d)=S^*(n,d-1)+(d-1)S^*(n,d),$ by examining the possibilities for inserting $(n+1)$ into a partition of $[n]$ into $d$ blocks. A comparison of the recurrences immediately shows that in fact 
  \[S^*(n+1,d)=S(n,d-1)\quad \text{for all } n\ge 0,d\ge 1.\] 
  See  \cite{Munagi} for generalisations of these numbers. %
  Recall that in Theorem~\ref{WHsubword}, the $j$th Whitney homology of $A^*_{n,k}$ $j\ge 2,$ was determined as a sum of two consecutive tensor powers of $S_{(n-1,1)}.$   From   Lemma~\ref{surj-normal-to-ptns}  and Proposition~\ref{LowerIntTopo} we now have the following surprising result.
  
  \begin{prop}\label{WhPermModule} Each Whitney homology module  of subword order, and hence the sum of two consecutive tensor powers of the reflection representation, has $h$-positive Frobenius characteristic, and in particular it is a permutation module.   We have  $\mathrm{ch\ }W\!H_0=h_n, \mathrm{ ch\ }W\!H_1=h_1 h_{n-1},$ and  for $k\ge j\ge 2,$ the $j$th Whitney homology of $A^*_{n,k}$ has Frobenius characteristic
  \begin{equation}\label{StirlingWh} (h_1h_{n-1})*s_{(n-1,1)}^{*(j-1)}=\sum_{d= 2}^j S(j-1,d-1)\, h_1^d h_{n-d}=\sum_{d=2}^j S^*(j,d)\,h_1^d h_{n-d}= h_1(h_1h_{n-2})^{*j-1},
  \end{equation}
a permutation module with orbits whose stabilisers are Young subgroups indexed by partitions of the form $(n-d,1^d), d\ge 0.$  

  \end{prop}
  
  \begin{proof}  From  Eqn.~\eqref{Whitney} and Theorem~\ref{WHsubword}, for $k\ge j\ge 2,$ we have 
  
  $W\!H_j(A^*_{n,k})= S_{(n-1,1)}^{\otimes j} \oplus S_{(n-1,1)}^{\otimes j-1}= (S_{(n)}\oplus S_{(n-1,1)})\otimes S_{(n-1,1)}^{\otimes j-1}.$ Now by definition we also have
  \[WH_j(A^*_{n,k})=\sum_{x\in A^*_{n,k}, |x|=j} \tilde{H}(\hat 0,x).\]
   Proposition~\ref{LowerIntTopo} says that the sum runs over only normal words $x,$ and each homology module is trivial for the stabiliser of $x.$ 
    Collecting the summands into orbits and using the surjection of Lemma~\ref{surj-normal-to-ptns}  gives Eqn.~\eqref{StirlingWh}.   The last expression is obtained by shifting the  index in the sum:
\[\sum_{d= 2}^j S(j-1,d-1)\, h_1^d h_{n-d}=h_1\sum_{d'=1}^{j-1} S(j-1,d') 
h_1^{d'} h_{(n-1)-d'},\]
and this equals $h_1(h_1h_{n-2})^{*j-1}$ by Lemma~\ref{Stirling}.
  \end{proof}
  Recall \cite{M} that the homogeneous symmetric functions $h_\lambda$ form a basis for the ring of symmetric functions.
  \begin{thm}\label{ReflRepMain} Fix $k\ge 1.$ 
  The $k$th tensor power of the reflection representation  
  $S_{(n-1,1)}^{\otimes k},$ i.e. the homology module $\tilde{H}_{k-1}(A^*_{n,k}),$ has the following property:
  $S_{(n-1,1)}^{\otimes k}\oplus (-1)^{k} S_{(n-1,1)}$ 
  is a permutation module $U_{n,k}$ whose Frobenius characteristic is $h$-positive, and  is supported on the set $\{h_\lambda: \lambda=(n-r, 1^r), r\ge 2\}.$  
  If $k=1,$ then $U_{n,1}=0.$
  
More precisely, the $k$-fold internal product $s_{(n-1,1)}^{ * k}$ has the following expansion in the basis of homogeneous symmetric functions $h_\lambda:$ 
  \begin{equation} \label{RefReptoh}
  \sum_{d=0}^n g_n(k,d)  h_1^d h_{n-d}, 
  \end{equation}
  where $g_n(k,0)=(-1)^k, g_n(k,1)=(-1)^{k-1},$ and
  \[g_n(k,d)=\sum_{i=d}^k (-1)^{k-i} S(i-1, d-1), \text{ for }2\le d\le n.\] 
  
  Hence  $s_{(n-1,1)}^{ * k}= (-1)^{k-1}s_{(n-1,1)}+\mathrm{ch}\,(U_{n,k}),$ where $\mathrm{ch}\,(U_{n,k})=\sum_{d=2}^n g_n(k,d)  h_1^d h_{n-d}.$
  
  The integers $g_n(k,d)$ are independent of $n$ for $k\le n,$  nonnegative for $2\le d\le k,$ and  $g_n(k,d)=0 $ if $d>k.$  
  Also:
  \begin{enumerate} 
  \item $g_n(k,2)=\frac{1+(-1)^k}{2}.$ %
  \item $g_n(k,k-1)=\binom{k-1}{2}-1, k\le n.$
  \item $g_n(k,k)=1, k< n.$
  \end{enumerate}
In particular the coefficient of $h_1^n$ in the expansion ~\eqref{RefReptoh} of $\mathrm{ch}\ S_{(n-1,1)}^{\otimes k}$ is 

  $\begin{cases} g_n(k,n)+g_n(k,n-1) &\mathrm{ if }\ k> n,\\
                        \binom{n-1}{2} &\mathrm{ if }\ k= n,\\
                         1 &\mathrm{ if }\ k=n-1,\\
                          0 &\mathrm{otherwise}. \end{cases}$
  \end{thm}
  
  \begin{proof}  If $k=1,$ the terms  for $d\ge 2$ in the summation in ~(\ref{RefReptoh}) vanish and thus the right-hand side equals the characteristic of  the top homology.
  
  The first statement, about the homology module $\tilde{H}_{k-1}(A^*_{n,k}),$  follows from Proposition~\ref{WhPermModule} and Theorem~\ref{WHmainSS}.
   Fix $m$ and $d$ such that $m\ge d\ge 2.$  Let $\bar g(m,d)$ be the alternating sum of Stirling numbers 
  $\bar g(m,d)=\sum_{i=d}^m (-1)^{m-i} S(i-1, d-1).$  Note that $\bar g(m,d)$ equals
  
  $[S(m-1,d-1)-S(m-2,d-1)]+[S(m-3,d-1)-S(m-4,d-1)] +\ldots $
  
 $\ldots +\begin{cases} [S(d+1,d-1)-S(d,d-1)] +S(d-1,d-1), & m-d\text{ even},\\
  [S(d,d-1)) -S(d-1,d-1)], & m-d\text{ odd}.
  \end{cases}$
  
  Since $S(n,d)$ is an increasing function of $n\ge d$ for fixed $d,$ the coefficient $\bar g(m,d)$ is always nonnegative.  It is also clear that $\bar g(m,d)= S(m-1,d-1) -\bar g(m-1,d)$ for all $m\ge d\ge 2.$
  
  The remaining parts follow from the facts that $S(k,k-1)=\binom{k}{2}, S(k,k)=1,$ and the observation that for $k\ge n,$ the coefficient of $h_1^n$ is $g_k(k,n)+g_k(k, n-1).$ This equals $S(n-1,n-2)-1$   when $k=n.$ \end{proof}

  \begin{cor}\label{TrivRepStableCase} Let $k\ge 2.$  
  \begin{enumerate}
  \item For $\min(n,k)\ge d\ge 2,$ the coefficient of $h_1^d h_{n-d}$ in $s_{(n-1,1)}^{* k}=\mathrm{ch}\, S_{(n-1,1)}^{\otimes k}$   is the nonnegative integer $g_n(k,d)$ given by the  two equal expressions:
  \begin{equation}\label{NewStirlingIdentity?}
  \sum_{j=d}^k (-1)^{k-j} S(j-1,d-1)=\sum_{r=0}^{k-d} (-1)^r \binom{k}{k-r} S(k-r,d).
  \end{equation}
  In particular, when $n\ge k,$ this multiplicity is independent of $n.$
  \item The positive integer $\beta_n(k)=\sum_{d=2}^{\min(n,k)} g_n(k,d)$ is the multiplicity of the trivial representation in $S_{(n-1,1)}^{\otimes k}.$  When $n\ge k,$ it equals the number of set partitions $B_k^{\ge 2}$ of the set $\{1,\ldots,k\}$ with no singleton blocks.  We have $\beta_n(n+1)=B_{n+1}^{\ge 2}-1$ and $\beta_n(n+2)= B_{n+2}^{\ge 2}- \binom{n+1}{2}.$
  \end{enumerate}
  \end{cor}
  
  \begin{proof}    
  %
  This  follows from Theorem~\ref{StirlingHom} and Corollary~\ref{TrivRep}.
  
  We have $\beta_n(n)=\sum_{d=2}^n g(n,d)=B_n^{\ge 2}= \beta_n(k)$ for $n\ge k,$ and  from \eqref{RefReptoh},
  
  $\beta_n(n+1)=\sum_{d=2}^{n}g_{n}(n+1,d)= \sum_{d=2}^{n+1} g_{n+1}(n+1,d) - g_{n+1}(n+1,n+1)= B_{n+1}^{\ge 2}-1,$
  
  $\beta_n(n+2)$
  
  $=\sum_{d=2}^{n+2} g_{n+2}(n+2,d) - g_{n+2}(n+2,n+2)
-g_{n+2}(n+2,n+1)$

$=B_{n+2}^{\ge 2} -1 - [\binom{n+1}{2}-1]=B_{n+2}^{\ge 2}- \binom{n+1}{2}.$
  \end{proof}

 We need one final observation in order to prove Theorem~\ref{AlmosthpositiveInstances}.
  \begin{lem}\label{RefRephPos} Suppose $V$ is an $S_n$-module which can be written as an integer combination $V=\oplus_{k=1}^m c_k S_{(n-1,1)}^{\otimes k}$ of positive tensor powers of $S_{(n-1,1)}.$ 
   Then \begin{enumerate}
   \item   The character value of $V$ on fixed-point-free permutations is  $\sum_{k=1}^m  (-1)^{k}c_k.$
   \item If
  $\sum_{k=1}^m  (-1)^{k-1}c_k=0,$ then the Frobenius characteristic of $V$ is %
  supported on the set $\{h_\lambda: \lambda=(n-r, 1^r), r\ge 2\}.$  
  \item If $\sum_{k=1}^m  (-1)^{k-1}c_k=0$ and $c_k\ge 0$ for all $k\ge 2,$ it is $h$-positive and hence $V$ is a permutation module.  %
  \end{enumerate}
  \end{lem}
  
  \begin{proof}  The first part follows because the value of the character of the reflection representation $S_{(n-1,1)}$ on permutations without fixed points is $(-1).$ 
  
 The second and third parts are immediate from Theorem~\ref{ReflRepMain}, since we have 
  \[V=(\sum_{k=1}^m (-1)^{k-1}c_k ) S_{(n-1,1)} \oplus \sum_{k=2}^m c_k U_{n,k}=\sum_{k=2}^m c_k U_{n,k} ,\]
  and $U_{n,k}$ is $h$-positive with support $\{h_\lambda: \lambda=(n-r, 1^r), r\ge 2\}.$  Note that $U_{n,1}=0.$

  \end{proof}
  
  In particular  from Theorem~\ref{Chains}, this gives  a direct proof that  the action of $S_n$ on the chains in $A_{n,k}^*$ is also a nonnegative linear combination of 
  $\{h_\lambda: \lambda=(n-r, 1^r), r\ge 2\}.$
  
%
 
\vskip.1in
\noindent
\textbf{Proof of Theorem~\ref{AlmosthpositiveInstances}:}
  \begin{proof}  Note that in all three cases, the  module $V_T=\tilde{H}_{k-2}(A^*_{n,k}(T))+ (-1)^{|T|} S_{(n-1,1)}$ has already been shown to 
  be an integer combination of $k$th tensor powers of $S_{(n-1,1)},$ with nonnegative  coefficients when $k\ge 2$,  in Theorems~\ref{Bordeaux1991mai}, ~\ref{rank-deletion} and ~\ref{TwoRanks}.    Hence, by  Lemma~\ref{RefRephPos}, it remains only to verify that the alternating sum of coefficients of the tensor powers vanishes for $V_T$ in each case.

  Consider  the case $T=[r,k]$. From Theorem~\ref{Bordeaux1991mai}, we must show that $(-1)^{k-r+1}$ added to the signed sum of the $(-1)^{i-1}b_i$, for the coefficients  $b_i=\binom{k}{i}\binom{i-1}{ k-r},$ is zero, i.e. 
  \begin{equation}\label{AnotherBinomialIdentity!} \sum_{i=1+k-r}^k b_i (-1)^{i-1} =(-1)^{k-r} .\end{equation}

 It is easiest to use the combinatorial identity of Corollary~\ref{NewBinomialIdentity?}.  Consider the two polynomials of degree $k\ge 2$ in $x$ defined by
 \[F(x)=  \sum_{i=0}^{k-r} (-1)^i {k\choose r+i}(x+1)^{r+i}x^{k-(r+i)}+(-1)^{k+1-r},\]
 \[G(x)= \sum_{i=1+k-r}^k {k\choose i}{i-1\choose k-r} x^i .\]
 Corollary~\ref{NewBinomialIdentity?} says $F(x)$ and $G(x)$ agree for all $x=n-1\ge 1,$ and hence $F(x)=G(x)$ identically.
 
 In particular $F(-1)=G(-1).$ But $F(-1)= (-1)^{k+r-1}$ and 
 clearly $(-1) G(-1)$ is precisely the expression in 
 ~\eqref{AnotherBinomialIdentity!}.  The claim follows.

  For the second case, $T$ is the rank-set $[1,k]\backslash\{r\},$ and from Theorem~\ref{rank-deletion} the alternating sum of coefficients in $V_T$ is clearly 
 \begin{equation}\label{rank-deletionIdentity}(-1)^{k-1}+ \left[{k\choose r}-1\right](-1)^{k-1}+ {k\choose r} (-1)^{k-2}=0.\end{equation}

For the third case, the rank set is $T=\{1\le s_1<s_2\le k\}.$  From the homology formula in Theorem~\ref{TwoRanks}, we need to show that the following sum, the alternating sum of coefficients in $V_T$, vanishes:
 \[1+\sum_{v=1}^{s_1} (-1)^{v-1} \left[c_v -\binom{s_1}{v}\right] + \sum_{v=1+s_1}^{s_2} (-1)^{v-1} c_v,\]
  where 
  \[c_v
  = \sum_{j=1}^{\min(v, s_2-s_1)}\binom{s_2-j}{v-j}\binom{s_1+j-1}{j} .\]
 But $1+\sum_{v=1}^{s_1} (-1)^{v-1} ( -\binom{s_1}{v})=0,$ so 
  this reduces to showing that 
$  \sum_{v=1}^{s_2} (-1)^{v} c_v=0.$
Split the summation over $v$ at $s_2-s_1.$ This gives that  $\sum_{v=1}^{s_2} (-1)^{v} c_v$ equals 

\[ \underbrace{\sum_{v=1}^{s_2-s_1}(-1)^{v}\sum_{j=1}^{v}\binom{s_2-j}{v-j}\binom{s_1+j-1}{j}}_{(A)}
+ \underbrace{\sum_{v>s_2-s_1}^{s_2}(-1)^{v}\sum_{j=1}^{ s_2-s_1}\binom{s_2-j}{v-j}\binom{s_1+j-1}{j}}_{(B)}.\]
Switching the order of summation,  $(A)$ is equal to 
\[\sum_{j=1}^{s_2-s_1} \binom{s_1+j-1}{j}\sum_{v=j}^{s_2-s_1} \binom{s_2-j}{v-j}(-1)^{v},\]
while (B) is
\[\sum_{j=1}^{s_2-s_1} \binom{s_1+j-1}{j}\sum_{v>s_2-s_1}^{s_2}(-1)^{v} \binom{s_2-j}{v-j}(-1)^{v}.\]
Hence $\sum_{v=1}^{s_2} (-1)^{v} c_v$ equals 
\[\sum_{j=1}^{s_2-s_1} \binom{s_1+j-1}{j}\sum_{v=j}^{s_2} \binom{s_2-j}{v-j}(-1)^{v}=\sum_{j=1}^{s_2-s_1} \binom{s_1+j-1}{j}\sum_{w=0}^{s_2-j} \binom{s_2-j}{w}(-1)^{w-s_2},\]
where we have put $w=s_2-v.$  But $1\le j\le s_2-s_1< s_2,$ so 
the inner sum vanishes.  \end{proof}

  Note that the left-hand side of Eqn.~\eqref{AnotherBinomialIdentity!} is the character value on fixed-point-free permutations for the homology module $\tilde{H}(T),$  $T=[r,k]$.  Hence this shows that the homology module itself  cannot be a permutation module when 
  $k-r$ is odd, since  the right-hand side then gives a value of $(-1)$ for the character.
  
  Similarly, from Theorem~\ref{rank-deletion}, the character value on fixed-point-free permutations for $\tilde{H}(T),$ for $T=[1,k]\backslash\{r\}$, 
  is $\left[{k\choose r}-1\right](-1)^{k}+ {k\choose r} (-1)^{k-1},$ and  this equals $(-1)^{k-1}.$  Once again we can conclude that this homology module is also not a permutation module when $k$ is even.
  
  \begin{cor} The dual Whitney homology modules $W\!H^*_{k+1-i}(A^*_{n,k}), 1<i\le k,$ are permutation modules whose  Frobenius characteristic is a nonnegative integer combination of the set $T_2=\{h_\lambda: \lambda=(n-r, 1^r), r\ge 2\}.$
  \end{cor}
  \begin{proof}   %
   From Theorem~\ref{WHmainSS} and ~\eqref{ConsecRanks}, we have 
  $W\!H_{k+1-i}^* = \tilde{H}([i,k]) \oplus \tilde{H}([i+1,k]).$ This in turn can be rewritten as 
  \[W\!H_{k+1-i}^* = \left( \tilde{H}([i,k])+(-1)^{k-i+1} S_{(n-1,1)} \right)\oplus \left( \tilde{H}([i+1,k]) +(-1)^{k-i} S_{(n-1,1)}\right);\]
  by Part 2 of Theorem~\ref{AlmosthpositiveInstances}, each of the summands in parentheses is $h$-positive with Frobenius characteristic supported by the 
   set $T_2.$
  \end{proof}
  
  \section{Tensor powers of the reflection representation II}
  In light of the preceding results, in this section we return to consider  the tensor powers $S_{(n-1,1)}^{\otimes k}$  more closely. The paper 
\cite{CG} gives a combinatorial model for determining the multiplicity of an irreducible in the $k$th tensor power, and an explicit formula in the case when $n$ is sufficiently larger than $k.$  We give general formulas that apply to the case of arbitrary tensor powers.

Burnside proved that given a faithful representation $V$ of a finite group $G$, every $G$-irreducible occurs in some tensor power of $V.$  A simple and beautiful proof of a generalisation of this was given by Brauer in \cite{BrauerR}. See also \cite[Theorem~4.3]{Isaacs}. In the present context, it  states that since $S_{(n-1,1)}$ is a faithful representation whose character takes on $n$ distinct values (viz. $-1,0,1,\ldots,n-1,$ but not $n-2$), every irreducible $S_\lambda$ occurs in at least one of the $n$  tensor powers $S_{(n-1,1)}^{\otimes j}, 0\le j \le n-1.$

In view of these results, the next fact is interesting.  We were unable to find it in the literature.  Our proof  mimics  Brauer's elegant argument.

\begin{thm}\label{myBrauer}  Let $G$ be any finite group and $X$ any character of $G$. Suppose $X$ takes on $k$ distinct nonzero values $b_1, \ldots, b_k.$
Then the first $k$ tensor powers of $X$ are linearly independent functions on $G$, and form a basis for the subspace of class functions spanned by all the positive tensor powers.   
If $X^{k+1}=\sum_{i=1}^k c_i X^i,$ then the polynomial 
$P(t)=t^{k+1}-\sum_{i=1}^k c_i t^i$ has the factorisation 
$t\prod_{i=1}^k (t-b_i).$ 

Let $X^0=1_G$ denote the trivial character of $G.$  Then
$X^{k}=\sum_{i=1}^{k} c_i X^{i-1}$ if and only if the character $X$ never takes the value zero.
\end{thm}

\begin{proof} %
 Let $U$ be the vector space spanned by the characters $X^j$ of the positive tensor powers of $X.$   Suppose $X$ takes the distinct values $\{b_i\ne 0: 1\le i\le k\}.$ 
 For each $i=1,\ldots , k,$ choose an arbitrary element in the preimage of $b_i,$ that is, $a_i\in X^{-1}(b_i).$ %
 Let $A=\{a_i: 1\le i\le k\}.$ We may thus view $U$ as a subspace of the space of functions defined on the set $A$ of size $k;$  this space has dimension exactly $k,$  and hence $\dim(U)\le k.$We will show that the characters $X^i, 1\le i \le k,$ are linearly independent.

 We now claim that the $k$ functions 
\[ \{X^j\downarrow_A, 1\le j\le k\}  \]
are linearly independent.
Suppose $c_j$ are scalars such that $\sum_{j=1}^{k} c_j X^j $ is the zero function.  This implies $\sum_{j=1}^{k} c_j X^j(a_j)=0. $
But $X^j(a_i)=b_i^j,$ so the coefficient matrix $(X^j(a_i))$ is a $k$ by $k$ Vandermonde  with  determinant $(b_1\ldots b_k)\prod_{1\le i<j\le k}(b_j-b_i),$ which is nonzero by hypothesis.  Hence $c_j=0$ for $ 1\le j\le k.$
This establishes the first statement.

Now let $X^{k+1}=\sum_{i=1}^k c_i X^i $ for some scalars $c_i.$
If $X$  never takes on the value zero we can clearly simplify the dependence relation to 
$X^{k}=\sum_{i=1}^k c_i X^{i-1}. $
If however 0 is a value of $X$, the set $\{1_G, X^i:1\le i\le k\}$ must be linearly independent, since the trivial character equals 1 everywhere.  This finishes the proof.
 \end{proof}

\begin{rk} For $k\ge 2,$ the representation $S_{(n-1,1)}^{\otimes (k-1)}$ is contained in $S_{(n-1,1)}^{\otimes k}.$  For this it suffices to note that 
\begin{enumerate}
\item  this is true for $k=2,$ since $S_{(n-1,1)}^{\otimes 2}- S_{(n-1,1)}=S_{(n-2,2)}\oplus S_{(n-1,1^2)}\oplus S_{(n)},$ and thus 

\item $S_{(n-1,1)}^{\otimes k}-S_{(n-1,1)}^{\otimes (k-1)}= 
S_{(n-1,1)}^{\otimes (k-2)}\otimes (S_{(n-1,1)}^{\otimes 2}-S_{(n-1,1)})$  is a true module.
\end{enumerate}
\end{rk}
  
   \begin{ex} Write $X_n^k$ for  $S_{(n-1,1)}^{\otimes k}.$ Maple computations  with Stembridge's SF package show that 
\begin{enumerate}
\item $X_3^3= X_3^2+2X_3.$
\item $X_4^4= 3X_4^3+X_4^2-3X_4.$
\item $X_5^5= 6X_5^4-7X_5^3-6X_5^2 +8 X_5.$
\item $X_6^6=10X_6^5 -30X_6^4+20X_6^3+31X_6^2-30X_6$
\item $X_7^7=15X_7^6 -79X_7^5 +165X_7^4 -64 X_7^3 -180 X_7^2+144 X_7$
\item $X_8^8=21 X_8^7 -168 X_8^6 +630 X_8^5 -1029 X_8^4 +189 X_8^3 +1198 X_8^2 -840 X_8.$
\end{enumerate}
\end{ex}
  We conclude this section by pointing out a representation-theoretic consequence, and some enumerative implications, of Theorem~\ref{ReflRepMain} and in particular of the expansion ~\eqref{RefReptoh}.
  Fix $n\ge 3$ and consider the $n$ by $n-1$ matrix $D_{n}$ whose $k$th column consists of the coefficients $g_n(n-k,n-d), d=1, \ldots, n-1.$  Thus the $k$th column contains the coefficients in the expansion of $S_{(n-1,1)}^{\otimes n-k}$ in the $h$-basis:  we have 
  $\mathrm{ch}\, S_{(n-1,1)}^{\otimes k}=\sum_{d=1}^n g_n(k, n-d)h_1^{n-d} h_d  , 1\le k\le n-1.$ From Theorem~\ref{ReflRepMain} it is easy to see that the matrix $D_{n}$ has rank $(n-1);$   the last two rows, consisting of alternating $\pm 1$s, differ by a factor of $(-1)$, and  the matrix is lower triangular with 1's on the diagonal, hence it has rank $(n-1).$  
  Similarly the $(n+1)$ by $(n-1)$ matrix obtained by appending to $D_{n} $ a first column consisting of the $h$-expansion of the $n$th tensor power of 
  $S_{(n-1,1)}^{\otimes n}$ also has rank $(n-1)$. We therefore have a second proof of Theorem~\ref{myBrauer}, for the specific case of the modules $S_{(n-1,1)}$. In this special case we can now be more precise about the linear combination of tensor powers:
  
  \begin{thm} The first $n-1$ tensor powers of $S_{(n-1,1)}$ are an integral basis for the vector space spanned by the positive tensor powers.  The $n$th tensor power of $S_{(n-1,1)}$ is an integer linear combination of the first $(n-1)$ tensor powers:
  \[S_{(n-1,1)}^{\otimes n}=\bigoplus_{k=1}^{n-1} a_k(n) S_{(n-1,1)}^{\otimes k},\] with $a_{n-1}(n)= \binom{n-1}{2}.$ 
  The coefficients $a_k(n)$ are determined by the polynomial 
   $P(t)=t^n -\sum_{k=1}^{n-1} a_k(n) t^{k},$  
  defined by
 \begin{equation}
  P(t)=\dfrac{t+1}{t-(n-2)}\sum_{j=1}^{n} c(n,j)t^{j} (-1)^{n-j} 
 \end{equation}
 where $c(n,j)$ is the number of permutations in $S_{n}$ with exactly $j$ disjoint cycles.
  
   \end{thm}
   
   \begin{proof}   We invoke Theorem~\ref{myBrauer}. The linear combination of tensor powers in the statement translates into a polynomial equation for the character values, whose zeros are the $n$  distinct values $-1, 0, 1, \ldots, n-3, n-1,$ taken by the character of $S_{(n-1,1)}$.  Hence we have 
    \[P_n(t)=t^n -\sum_{k=1}^{n-1} a_k(n) t^{k}=(t+1)t
  \prod_{i=1, i\ne n-2}^{n-1}(t-i)=\frac{t+1}{(t-(n-2))}\prod_{i=0}^{n-1}(t-i).\]
  But $\prod_{j=0}^{n-1} (t-j) $ is the generating function for the signless Stirling numbers of the first kind \cite{St3EC1}, so the result follows. 
   \end{proof}

The preceding result gives a recurrence for the coefficients $a_k(n)$;
we have
\begin{align*} a_{n-1}(n)&=\binom{n-1}{2};\\
(n-2)a_j(n)-a_{j-1}(n)&=(-1)^{n-j}[c(n,j)-c(n,j-1)], 2\le j\le n-1;\\
(n-2) a_1(n)&= c(n,1) (-1)^{n-1}\\
\Longrightarrow a_1(n)&= \frac{(n-1)!}{n-2} (-1)^{n-1} =(-1)^{n-1}[(n-2)!+(n-3)!]
\end{align*}

 \begin{qn} The identity ~\eqref{NewStirlingIdentity?} holds for all $d=2,\ldots,k.$ 
 Is there a combinatorial explanation?
 \end{qn}
 
  \begin{qn} For fixed $k$ and $n,$ what do the positive integers $g_n(k,d)$ count?   Is there a combinatorial interpretation for $\beta_n(k)=\sum_{j=d}^{\min(n,k)}g_n(k,d),$  the multiplicity of the trivial representation in the top homology of $A^*_{n,k},$  in the nonstable case $k>n?$  Recall that for $k\le n$ this is the number $B_k^{\ge 2}$ of set partitions of $[k]$ with no singleton blocks, and is sequence OEIS A000296.  
 \end{qn}
 
 \begin{qn} Recall that $a_{n-1}(n)=\binom{n-1}{2}.$ Is there a combinatorial interpretation for the signed integers $a_i(n)$?  There are many interpretations for $(-1)^{n-1} a_1(n)= (n-2)!+(n-3)!,$ see OEIS A001048.  For $n\ge 4$ it is the size of the largest conjugacy class in $S_{n-1}.$  We were unable to find the other sequences $\{a_{i}(n)\}_{n\ge 3} $ in OEIS.
 \end{qn}

  \section{The subposet of normal words}
  
  Let $N_{n,k}$ denote the poset of normal words of length at most $k$ in $A^*_{n,k},$ again with  an artificial top element $\hat 1$ appended. Farmer  showed that
  \begin{thm} (Farmer \cite{F})  $\mu(N_{n,k})=(-1)^{k-1} (n-1)^{k} =\mu(A^*_{n,k}),$ and $A_{n,k}, N_{n,k}$ both have the homology of a wedge of $(n-1)^{k}$  spheres of dimension  $(k-1)$.
  \end{thm}
  
  Bj\"orner and Wachs \cite{BjWachs} showed that $N_{n,k}$ is dual CL-shellable and hence homotopy Cohen-Macaulay; it is therefore homotopy-equivalent to a wedge of $(n-1)^{k}$ spheres of dimension  $(k-1)$.
  The order complexes of the posets $A^*_{n,k}$ and $N_{n,k}$ are thus homotopy-equivalent.  
  
  Using Quillen's fibre theorem (\cite{Q})  we can establish a slightly stronger result:
  
  \begin{lem}  Let $\alpha\in N^*_{n,k}.$ Then  the intervals $(\hat 0,\alpha)_{N_{n,k}}$ and  $(\hat 0,\alpha)_{A^*_{n,k}}$ are ${\textrm{stab}(\alpha)}$-homotopy equivalent,  for the stabiliser subgroup ${\textrm{stab}(\alpha)}$ of $\alpha.$ 
  In particular the homology groups are all ${\textrm{stab}(\alpha)}$-isomorphic. 
  
  If $\alpha\in A^*_{n,k},$ but $\alpha\notin N^*_{n,k},$ then we know that the interval $(\hat 0, \alpha)_{ A^*_{n,k}}$ is contractible.
  \end{lem}
  
  \begin{proof}  Let $J_m$ be the set of words of length $m$ that are \textit{not} normal.  Let 
  \[B_j=(\hat 0, \alpha)_{A^*_{n,k}}\backslash (\cup_{m=j}^k J_m) \] be the subposet obtained by removing all normal words at rank $j$ and higher.  Thus $B_1=(\hat 0, \alpha)_{N_{n,k}}.$  Set $B_{k+1}=(\hat 0, \alpha)_{A^*_{n,k}}.$  We claim that  the inclusion maps
  \begin{equation}
(\hat 0, \alpha)_{N_{n,k}}= B_1\subset B_2\subset \ldots B_j\subset B_{j+1}\subset B_{k+1}=(\hat 0, \alpha)_{A^*_{n,k}}
  \end{equation}
  are group equivariant homotopy equivalences.
  Note that 
  \[B_j=B_{j+1}\backslash \{\text{non-normal words of length } j+1\},\] and $B_j$ coincides with $B_{j+1}$ for the first $j$ ranks.  The fibres to be checked are $F_{\le w}=\{\beta\in B_j: \beta\le \alpha\},$ for $w\in B_{j+1}.$  If $w $ is a normal word in $B_{j+1},$ then $w\in B_j$ and  the fibre is the half-closed interval $(\hat 0, w]$ in $B_j; $ it is therefore contractible.  If $w\in B_{j+1}$ is not a normal word, then $w\notin B_j$ and the interval $(\hat 0, w)_{B_j}$ coincides with the same interval in $A^*_{n,k},$ so by Part (1) of Proposition~\ref{LowerIntTopo}, it is contractible.  Hence by Quillen's fibre theorem, the inclusion induces a homotopy equivalence.  
  \end{proof}
  
  \begin{prop}\label{WHnormal}  The Whitney homology modules of $A^*_{n,k}$ and $N_{n,k}$ are $S_n$-isomorphic
  In particular the $j$th Whitney homology of $N_{n,k}$ is  isomorphic to the $S_n$-action on the elements at rank   $j.$ 
  \end{prop}
  
  This statement is false for the dual Whitney homology. For instance, $\mu(ab, abab)_{A^*_{n,k}}=+3,$ but $\mu(ab, abab)_{N_{n,k}}=+1.$ The first interval is a  rank-one poset with four elements $bab, aab, abb, aba,$ whereas the second consists only of two elements $aba, bab.$ 
  
  It is also easy to find examples showing that the rank-selected homology is not the same for each poset.  
  
  \begin{ex}
 Let $n=2$ and consider the rank-set $\{1,3\}$ for the poset $A^*_{2,k}$ and for its subposet of normal words $N_{2,k}$.
 
 The words of length 3, all of which cover the two rank 1 elements $a$ and $b,$ are 
 \[ aaa, aab, aba, abb, baa, bab, bba, bbb.\]
 It is clear that the M\"obius function values are 
 \[\mu(\hat 0, aaa)=0=\mu(\hat 0, bbb), \mu(\hat 0, w)=-1 \text{ for all } w\notin \{aaa, bbb\}.\]
 Hence the M\"obius number of the rank-selected subposet of all words is $-5,$ and from Theorem~\ref{rank-deletion} the $S_2$-representation on homology is $3\, S_{(2)}\oplus 2 \, S_{(1,1)}.$ The order complex is a wedge of 5 one-dimensional spheres.
 
 Now consider the corresponding rank-selected subposet of normal words:
 there are only two normal words of length 3, namely $aba, bab$ and hence the M\"obius number of the rank-selected subposet of normal words is $-1,$ with trivial homology representation.  The order complex is a  one-dimensional sphere.
 \end{ex}
 
  \begin{ex}
 More generally,  let $S$ be the rank-set $[2,k]$, and consider the posets  $A_{2,k}(S)$ and $N_{2,k}(S)$ obtained by deleting the atoms.  Then by Theorem~\ref{rank-deletion} the homology of $A_{2,k}(S)$ is 
 $ (k-1)S_{(1^2)}^{\otimes k} +k S_{(2)},$
 while the homology of the normal word subposet $N_{2,k}(S)$ is seen to be 
  $S_{(1^2)}^{\otimes k},$ which is either the trivial  or the sign module, depending on the parity of $k.$ 
 \end{ex}
 \begin{rk}
 In fact it is easy to see that $N_{2,k}$ is the ordinal sum \cite{St3EC1} of $k$ copies of an antichain of size 2, with a bottom and top element attached.  Hence for any subset $T$ of $[1,k]$, there is an $S_2$-equivariant poset isomorphism between $N_{2,k}(T)$ and $N_{2,|T|}.$ Since the $S_n$-homology of $N_{2,k}$ is easily seen to be the 
 $k$-fold tensor power of the sign representation, 
 this determines $\tilde{H}(N_{2,k}(T))$ for all rank-sets $T.$
 Also,  $S_2$ acts  on the chains of any rank-selected subposet $N_{2,k}(T)$ like $2^{|T|-1}$ copies of the regular representation of $S_2,$ since the $2^k$ chains of $N_{2,k}$ break up into orbits of size 2.
 \end{rk}
 
 Recall from \cite{St3EC1} that a finite graded poset $P$ with $\hat 0$ and $\hat 1$ is Eulerian if its M\"obius function $\mu_P$ satisfies $\mu(P)=(-1)^{\mathrm{rank}(y)-\mathrm{rank}(x)}$ for all intervals $(x,y)\subseteq (\hat 0, \hat 1).$
 It is known that all intervals $(x,y),  y\neq \hat 1,$ in $N_{n,k}$ are Eulerian (see e.g. \cite[Exercise 188]{St3EC1}). In fact Bj\"orner and Wachs observed in \cite{BjWachs} that for a finite alphabet $A=\{a_i: 1\le i\le n\},$ the poset of normal words without the top element
  $N_{n,k}\backslash\{\hat 1\}$  is simply Bruhat order on the Coxeter group with $n$ generators $a_i$ and relations $a_i^2=1.$ Thus by lexicographic shellability \cite{BjWachsAIM1982}, all intervals $(x,y),  y\neq \hat 1$ are homotopy equivalent to a sphere.

 In fact this makes $N_{n,k}\backslash\{\hat 1\}$ a $CW$-poset as defined in \cite{Bj-CW}, and hence its order complex is isomorphic to the face poset of a regular $CW$-complex $\mathcal{K}$ \cite[Proposition~3.1 ]{Bj-CW}.  It is easy to see from the definitions that if a finite group $G$ acts on a $CW$-poset $P$, then there is a $G$-module isomorphism between the $j$th Whitney homology $W\!H_j(P)$ and the $G$-action on the $(j-1)$-cells of the associated  
 $CW$-complex $\mathcal{K}(P)$, which are simply the elements of $P$ at rank $j$.  This provides another explanation for the observation of Proposition~\ref{WHnormal}.

  An EL-labelling of the dual poset of normal (or Smirnov) words appears in \cite{TiansiLi}.  In \cite{LiSu} the program of the present paper is carried out for the subposet of Smirnov words.


\bibliographystyle{amsplain.bst}

\end{document}